\documentclass{article}

\usepackage[a4paper,bottom=2cm,top=2cm,left=2.54cm,right=2.54cm]{geometry}
\usepackage[english]{babel}
\usepackage[utf8]{inputenc}
\usepackage{url}
\usepackage[usenames]{color}
\usepackage{amsmath}
\usepackage{amssymb}
\usepackage{amsthm}
\usepackage{bbm}
\usepackage{color}
\usepackage{soul}
\usepackage{graphicx}
\usepackage{subcaption}
\usepackage{yfonts}
\usepackage[backend=bibtex,maxnames=5,url=false,style=numeric,isbn=false, giveninits=true,eprint=false,sorting=none,doi=false,date=year]{biblatex}
\renewbibmacro{in:}{}
\usepackage{alphalph}
\usepackage{enumitem}
\usepackage{courier}
\usepackage[dont-mess-around]{fnpct}
\usepackage{etoolbox}

\newcommand{\defeq}{:=}
\newcommand{\norm}[1]{\| #1\|}

\newcommand{\abs}[1]{|#1|}

\newcommand{\one}{\mathbbm 1}

\renewcommand{\leq}{\leqslant}
\renewcommand{\geq}{\geqslant}
\renewcommand{\phi}{\varphi}
\newcommand{\qand}{\quad \text{and} \quad}

\newcommand{\mysum}[3]{\sum_{#1 = #2}^{#3}}

\newcommand{\subdiff}{\partial}

\newcommand{\reg}{\mathcal J}
\newcommand{\dJ}{\subdiff \reg}

\newcommand{\Jminsol}{u^\dagger}
\newcommand{\unJ}{u^{n,\delta}_\reg}

\let\span\relax
\DeclareMathOperator{\span}{Span}
\let\ker\relax
\newcommand{\ker}[1]{\mathcal N(#1)}
\newcommand{\range}[1]{\mathcal R(#1)}

\newcommand{\cl}[1]{{\overline{#1}}}
\newcommand{\orth}[1]{#1^\perp}

\DeclareMathOperator{\BV}{BV}
\DeclareMathOperator{\TV}{TV}

\newcommand{\R}{\mathbb R}
\newcommand{\N}{\mathbb N}

\newcommand{\weakto}{\mathrel{\rightharpoonup}}

\newtheorem{theorem}{Theorem}

\newtheorem{proposition}[theorem]{Proposition}

\theoremstyle{definition}
\newtheorem{definition}[theorem]{Definition}
\newtheorem{assumption}{Assumption}
\newtheorem*{assumption*}{Assumption}
\newtheorem{remark}[theorem]{Remark}
\newtheorem*{remark*}{Remark}
\newtheorem*{definition*}{Definition}
\newtheorem{example}[theorem]{Example}

\newcommand{\orthdata}{\underline}
\newcommand{\orthim}{\overline} 
\newcommand{\U}{\mathcal U}
\newcommand{\Y}{\mathcal Y}

\newcommand{\Radon}{\mathcal R}
\newcommand{\ind}[3]{{#1=#2,\dots,#3}}
\newcommand{\ui}{u^i}
\newcommand{\yi}{y^i}

\newcommand{\udn}{u^\U_n}
\newcommand{\udndelta}{u^\U_{n,\delta}}
\newcommand{\vdndelta}{u^\Y_{n,\delta}}
\newcommand{\udndeltan}{u^\U_{n(\delta),\delta}}
\newcommand{\vdn}{u^\Y_n}

\newcommand{\trim}{\{u^i\}_\ind{i}{1}{n}}
\newcommand{\trdata}{\{y^i\}_\ind{i}{1}{n}}
\newcommand{\trpairs}{\{u^i, y^i\}_\ind{i}{1}{n}}

\newcommand{\trimhat}{\{\orthim u^i\}_\ind{i}{1}{n}}
\newcommand{\trdatahat}{\{\orthim y^i\}_\ind{i}{1}{n}}
\newcommand{\trpairshat}{\{\orthim u^i, \orthim y^i\}_\ind{i}{1}{n}}

\newcommand{\trimb}{\{\orthdata u^i\}_\ind{i}{1}{n}}
\newcommand{\trdatab}{\{\orthdata y^i\}_\ind{i}{1}{n}}
\newcommand{\trpairsb}{\{\orthdata u^i, \orthdata y^i\}_\ind{i}{1}{n}}

\newcommand{\triminf}{\{u^i\}_{i \in \N}}
\newcommand{\trdatainf}{\{y^i\}_{i \in \N}}

\newcommand{\trimhatinf}{\{\orthim u^i\}_{i \in \N}}
\newcommand{\trdatahatinf}{\{\orthim y^i\}_{i \in \N}}

\newcommand{\trimbinf}{\{\orthdata u^i\}_{i \in \N}}
\newcommand{\trdatabinf}{\{\orthdata y^i\}_{i \in \N}}

\newcommand{\ds}{``Faces'' }
\providecommand{\keywords}[1]{\textbf{Keywords: } #1}
\makeatletter
\renewcommand*{\@fnsymbol}[1]{\@arabic{#1}}
\makeatother

\makeatletter
\patchcmd\maketitle{\setcounter{footnote}{0}}{}{}{}
\patchcmd\maketitle{%
  \renewcommand\thefootnote{\@fnsymbol\c@footnote}}{\AdaptNote\thanks\multthanks}{}{}
\patchcmd\maketitle{%
  \def\@makefnmark{\rlap{\@textsuperscript{\normalfont\@thefnmark}}}}{}{}{}
\makeatother

\numberwithin{equation}{section}

\usepackage{tikz, pgfplots}
\pgfplotsset{compat=newest}
\pgfplotsset{plot coordinates/math parser=false}
\usetikzlibrary{external}
\usepackage[textsize=footnotesize,colorinlistoftodos]{todonotes}

\pdfoutput=1

\newcommand{\bb}[1]{{#1}}
\newcommand{\rb}[1]{{#1}}

\bibliography{abbrevs.bib,IP.bib,BL.bib,datasets.bib}

\title{Data driven regularization by projection}

\date{}

\author{Andrea Aspri\footnotemark[1]
\and Yury Korolev\footnotemark[2]\footnotemark[4]
\and Otmar Scherzer\footnotemark[3]\footnotemark[1]}

\begin{document}
\maketitle

\renewcommand{\thefootnote}{\fnsymbol{footnote}}
\footnotetext[1]{Johann Radon Institute for Computational and Applied Mathematics, Altenberger Stra{\ss}e 69, 4040 Linz, Austria, \mbox{andrea.aspri@ricam.oeaw.ac.at}}
\footnotetext[2]{Department of Applied Mathematics and Theoretical Physics, University of Cambridge, Wilberforce Road, Cambridge CB3 0WA, UK, \mbox{y.korolev@damtp.cam.ac.uk}}
\footnotetext[3]{Department of Mathematics, University of Vienna, Oskar-Morgenstern-Platz 1, 1090 Vienna, Austria, \mbox{otmar.scherzer@univie.ac.at}}
\footnotetext[4]{Communicating author}

\abstract{We study linear inverse problems under the premise that the forward operator is not at hand but given indirectly 
	through some input-output training pairs. We demonstrate that \emph{regularisation by projection} and 
	\emph{variational regularisation} can be formulated by using the training data only and without making use of the forward operator. 
	We study convergence and stability of the regularised solutions in view of 
	\fullcite{Seidman:1980}, 
	who showed that regularisation by projection is \emph{not convergent} in general, by giving some insight 
	on the generality of Seidman's nonconvergence example.
	Moreover, we show, analytically and numerically, that regularisation by projection is indeed capable of \emph{learning} linear operators, such as the Radon transform.
}

\keywords{Data driven regularisation, variational regularisation, regularisation by projection, inverse problems, Gram-Schmidt orthogonalisation}

\section{Introduction}
\emph{Linear inverse problems} are concerned with the reconstruction of a quantity $u \in \U$ from 
indirect measurements $y \in \Y$ which are related by the linear \emph{forward operator} $A \colon \U \to \Y$.
$A$ models the physics of data acquisition and may involve, for instance, integral transforms (such as the Radon 
transform, see for instance \cite{Natterer_Wubbelling,Sch15}) and partial differential equations 
(PDEs) (see for instance~\cite{isakov_IP_PDEs}). 

Until recently the methods for solving inverse problems were \emph{model driven}, meaning that the physics and 
chemistry of the measurement acquisition process was represented mathematically by the forward operator $A$ 
(for some relevant applications see \cite{Natterer_Wubbelling,Ellerbroek_2009,Symes_2009}). 
Nowadays, with the rise of the area of big data, methods that combine forward modelling with data driven techniques are being developed~\cite{arridge_et_al_acta_numerica}. 
Some of these techniques build upon the similarity between deep neural networks and classical approaches to inverse problems such as iterative regularisation~\cite{Adler_2017,AspSch19_report} and proximal methods~\cite{Pock_variational:2017}. Some are based on postprocessing of the reconstructions obtained by a simple inversion technique such as filtered backprojection~\cite{Unser_denoising:2017}. Others use data driven regularisers in the context of variational regularisation~\cite{Haltmeier_NETT, Lunz_CBS_adversarial} or use deep learning to learn a component of the solution in the null space of the forward operator~\cite{Haltmaier_deep_nullspace:2019,Bubba_Kutyniok_invisible:2019}.

On the other hand, methods that discard forward modeling, that means knowledge of the operator $A$, have emerged. 
They are appealing because they do not require knowledge on the physics and chemistry of the data acquisition process, 
bypass the costly forward model evaluation and often yield results of superior visual quality. 
However, it has been demonstrated that for ill-posed inverse problems na\"ive applications of such methods can be unstable with 
respect to small perturbations in the measurement data~\cite{hansen_ai_at_a_cost, maass_trivial_IP:2019}. 
Moreover, there is currently no theory for purely data driven regularisation in inverse problems, i.e. a theory in the setting 
when the forward operator is given only via training pairs 
\begin{equation}\label{eq:pairs}
\{\ui, \yi \} \quad \text{such that} \quad A\ui= \yi.
\end{equation}
In this paper we make a first step of an analysis for purely data driven regularisation. 
Our goal is to develop and analyse stable algorithms for solving $Au=y$ \emph{without} making use of the 
explicit knowledge of the forward operator $A$ but having only
\begin{itemize}
	\item approximate measurement $y^\delta$ such that $\norm{y-y^\delta} \leq \delta$ and 
	\item training pairs~\eqref{eq:pairs}.
\end{itemize}
In particular, we demonstrate that regularisation by projection~\cite{Seidman:1980,engl:1996} and  variational regularisation~\cite{scherzer_var_meth:2009}  can be formulated in a data driven setting. We provide new analysis of regularisation by projection and derive sufficient conditions for convergence that shed some new light on the classical nonconvergence example in~\cite{Seidman:1980}.  
Finally, in the numerical part of the paper we demonstrate that the methods we propose yield similar results to classical methods and require moderate amounts of training data.
Note that we consider the idealised case when the training pairs \eqref{eq:pairs} are noise-free. We leave out for future research the natural extension to the case of noisy training data. 

There is a whole body of literature concerning regularisation with various types of projections, typically onto subspaces spanned by some predefined bases such as polynomials or wavelets; a complete list of references would be too long to be presented here. Our approach differs from these works in the fact that our subspaces are defined by training data. This presents challenges since there is typically no relation between these subspaces and the properties of the forward operator (e.g., its eigenspaces) or the properties of the exact solution such as smoothness. It is hard, therefore, to make claims about optimality of the methods.

A connection between regularisation by projection and machine learning has been observed by several authors, for instance, in the context of statistical learning~\cite{rudi_camoriano_rosasco:2015, kriukova_pereverzyev_tkachenko:2017}, 
where it was demonstrated that \emph{subsampling} plays the role of regularisation by (random) projections. 
From the statistical point of view, the idea to approximate an operator from input-output samples is closely related to statistical numerical
approximation (see, e.g.,~\cite{Owhadi:2019} for a recent review) and in particular to optimal recovery~\cite{Micchelli_Rivlin:1977}. 

\paragraph{Structure of the paper.} In Section~\ref{sec:setting} we summarise our assumptions and describe our setting and notation in more detail. Regularisation by projection onto the space of input data is the topic of Section~\ref{sec:reg_proj}. We derive sufficient conditions on the input data and the outputs data, under which projections result in a regularisation. In Section~\ref{sec:dual_lsq} we present another regularisation method based on projections onto the space of output data, the \emph{dual least squares} method~\cite{engl:1996}, which is known to be convergent. We demonstrate, however, that it cannot be realised using the training pairs~\eqref{eq:pairs} and requires a different type of training data. In Section~\ref{sec:var_reg} we study variational regularisation and in Section~\ref{sec:numerics} we present numerical experiments with all three methods. We show that regularisation by projection is indeed capable of \emph{learning} linear operators, such as the Radon transform. 

\paragraph{Main contributions.} \rb{Our main goal in this article is to carry over some classical results in regularisation theory to the model-free setting, with an overarching theme of using projections on subspaces defined by training data in the framework of regularisation by projection. This perspective requires new, data driven regularity conditions, for which we show a relation to source conditions in special cases, whilst in general such relationship remains an open question. A classical nonconvergence result of Seidman~\cite{Seidman:1980} demonstrates the sharpness of this regularity condition.
}

\rb{We also demonstrate that the amount of training data (which is also related the complexity of the learned model in regularisation by projection) plays the role of a regularisation parameter, hence for noisy training data the size of the training set should be chosen in agreement with the noise level. This is in accordance with the results in~\cite[Thm. 4.2]{Burger_Engl:1999} on training neural networks from noisy data, where the number of neurons in the network plays a role similar to the number of training inputs in our setting.}


\section{Setting and Main Assumptions}\label{sec:setting}
We consider a linear inverse problem, consisting in solving 
\begin{equation}\label{eq:Au=y}
Au = y,
\end{equation}
with a linear bounded forward operator $A \colon \U \to \Y$ acting between separable Hilbert spaces $\U$ and $\Y$.
Different to the standard setting of linear inverse problems (see for instance \cite{Groetsch:1984}) we do not make use of 
the knowledge of the operator $A$ but only assume knowledge of training pairs \eqref{eq:pairs}. 
This is beneficial if the modeling of the forward operator is rather incomplete, 
very uncertain, or the numerical evaluation is costly.

We assume that $A$ is injective and its inverse $A^{-1}$ is unbounded, hence the problem of solving~\eqref{eq:Au=y} is ill-posed. 
Instead of the exact measurement $y$ we are given a noisy measurement $y^\delta$, satisfying 
\begin{equation} \label{eq:delta}
\norm{y-y^\delta} \leq \delta.
\end{equation}
Throughout this paper, the solution of \eqref{eq:Au=y} with \emph{exact data} $y \in \range{A}$ will be denoted by $u^\dagger$.

We refer to elements of $\U$ as 
\emph{inputs} 
and to elements of $\Y$ as 
\emph{outputs}. 
Accordingly, speaking of the training pairs~\eqref{eq:pairs}, we will refer to $\trim$ as {\emph{training inputs}} and to $\trdata$ as {\emph{training outputs}}.

We make the following assumptions on the training pairs throughout this paper.
\begin{assumption}[Independence, uniform boundedness, sequentiality]\label{ass_1} 
	 \ \\
	\textit{Linear independence:} For every $n \in \N$ the inputs $\trim$ are linearly independent.\\
	\textit{Uniform boundedness:} There exist constants $c_u,C_u>0$ such that $c_u \leq \norm{\ui} \leq C_u$  for all $i \in \N$. Hence with no loss of generality we will assume that $\norm{\ui} =1$ for all $i \in \N$. \\
	\textit{Sequentiality:} The families of training pairs are nested, i.e. for every $n \in \N$ 
	\begin{equation}\label{ass_2} 
	\{\ui, \yi\}_\ind{i}{1}{n+1} = \trpairs \cup \{u^{n+1}, y^{n+1}\}.
	\end{equation}
\end{assumption}

\rb{If the training inputs $\trim$ are linearly dependent, we will need to discard part them. Which ones to discard is an interesting question in itself, and a possible topic for future work.}

\begin{definition} \label{def:Un}
	We denote the spans of the inputs $\trim$ and the outputs $\trdata$ by
	\begin{equation}
	\U_n \defeq \span\trim \qand \Y_n \defeq \span\trdata.
	\end{equation}
	Orthogonal projection operators onto $\U_n$ and $\Y_n$ are denoted by $P_{\U_n}$ and $P_{\Y_n}$, respectively.
\end{definition}

\begin{remark}\label{prop:data_lin_indep}
	From the Assumption \ref{ass_1} and the hypotheses made on $A$ we deduce that:
	\begin{enumerate} 
		\item Since the forward operator $A$ is injective, the outputs $\trdata$ are also linearly independent. 
		\item Since $A$ is a bounded operator and the inputs $\trim$ are uniformly bounded, the outputs $\trdata$ are also uniformly bounded and $\norm{y^i}\leq \norm{A}$, for all $i\in \N$.
	\end{enumerate}
	From Assumption~\ref{ass_1} and \eqref{ass_2}, it follows that the subspaces $\U_n$, $\Y_n$ are nested, that is 
	\begin{equation*}
	\U_n \subset \U_{n+1}, \quad \Y_n \subset \Y_{n+1}, \quad \text{for all $n$.}
	\end{equation*}
\end{remark}
We also need to make an assumption that the training data are sufficiently rich in the sense  that the collection of all inputs $\trim$ is dense in $\U$.
\begin{assumption}[Density]\label{ass_3}
	We assume that the subspaces spanned by the inputs $\trim$ are dense in $\U$, that is 
	\begin{equation*}
	\cl{\bigcup_{n \in \N} \U_n} = \U. 
	\end{equation*}
\end{assumption}
As a consequence of the previous assumptions, we have
\begin{proposition} \label{pr:dense_Y}  
	By Assumption~\ref{ass_3} the subspaces spanned by the training outputs $\trdata$ are dense in $\cl{\range{A}}$, i.e.,
	\begin{equation*}
	\cl{\bigcup_{n \in \N} \Y_n} = \cl{\range{A}}. 
	\end{equation*}
\end{proposition}

\bb{We carry out our analysis in the setting when the training pairs $\trpairs \subset \U\times\Y$ are infinite-dimensional. The case when these pairs are discretised, i.e. $\trpairs \subset \R^k\times\R^l$ and $k,l,n \to \infty$ at certain rates is beyond the scope of this paper.}

\section{Regularisation by Projection} \label{sec:reg_proj}
Let $y \in \range{A}$ be the exact, noise-free right-hand side in~\eqref{eq:Au=y} and consider the following projected problem
\begin{equation} \label{AP_n u = y}
AP_{\U_n} u = y.
\end{equation}
Our goal in this section is to provide an explicit representation formula for the minimum norm solution of ~\eqref{AP_n u = y} in terms of the training pairs \eqref{eq:pairs}. We first observe that the minimum norm solution of ~\eqref{AP_n u = y} is given by
\begin{equation}\label{eq:udn1}
\udn = (AP_{\U_n})^\dagger y,
\end{equation}
where $(AP_{\U_n})^\dagger$ denotes the Moore-Penrose inverse of $AP_{\U_n}$ (see \cite{Nas76}). The superscript $^{\U}$ in $\udn$ reflects the fact that the projection in~\eqref{AP_n u = y} takes place in $\U$. The need for this notation will become clear in Section~\ref{sec:dual_lsq}, where will use a projection of the original equation~\eqref{eq:Au=y} in the space $\Y$.

\subsection{A reconstruction formula}
\label{sec:rec_formula}
The following result shows that in the injective case, a simple formula for $(AP_{\U_n})^\dagger$ exists.

\begin{theorem}\label{thm:pseudoinv_image}
	Let  $P_{\U_n}$ and $P_{\Y_n}$ as in Definition~\ref{def:Un}. 
	 Then the Moore-Penrose inverse of $AP_{\U_n}$ is given by
	\begin{equation*}
	(AP_{\U_n})^\dagger = A^{-1} P_{\Y_n}.
	\end{equation*}
\end{theorem}
\begin{proof}
	First we observe that $\range{A^{-1} P_{\Y_n}} = \U_n = \orth{\ker{AP_{\U_n}}}$. To see this, note that 
    for any $z \in \Y$ we have
    \begin{equation*}
       P_{\Y_n} z = \sum_{i=1}^n \lambda_i^n \yi 
    \end{equation*}   
    (where the expansion coefficients, in general, change with $n$) and hence for $u = A^{-1}P_{\Y_n} z$ we have by definition 
    \begin{equation*}
      u=A^{-1} \left( \sum_{i=1}^n \lambda_i^n \yi \right) =   \sum_{i=1}^n \lambda_i^n \ui \in \U_n.
    \end{equation*}
    On the other hand, for any $u \in \U_n$ we can also find a representation in the form $u = A^{-1}P_{\Y_n} y$ for some $y \in \Y$. Thus the first identity is shown. For the second identity, let $v \in \ker{AP_{\U_n}}$. Since $A$ is injective, this is equivalent to $v \in \ker{P_{\U_n}} = {\U_n^\bot}$ and hence $\ker{AP_{\U_n}} = \U_n^\bot$. Since $\U_n$ is finite-dimensional and hence closed, this implies the second identity.
	
	Using the obvious identities
	\begin{equation*}
	P_{\Y_n}AP_{\U_n} = AP_{\U_n} \quad \text{and} \quad P_{\U_n} A^{-1} P_{\Y_n} = A^{-1}P_{\Y_n},
	\end{equation*}
	we directly verify the Moore-Penrose equations (see for instance \cite{Nas76}):
	\begin{itemize}
		\item $AP_{\U_n} A^{-1} P_{\Y_n} AP_{\U_n} = A P_{\U_n} A^{-1} AP_{\U_n} = AP_{\U_n}$;
		\item $A^{-1} P_{\Y_n} A P_{\U_n} A^{-1} P_{\Y_n} = A^{-1} P_{\Y_n} A A^{-1} P_{\Y_n}  = A^{-1} P_{\Y_n}$;
		\item $A^{-1} P_{\Y_n} AP_{\U_n} = A^{-1} AP_{\U_n} = P_{\U_n} = I - P_{\U_n^\perp} = I-P_{\ker{AP_{\U_n}}}$;
		\item $AP_{\U_n} A^{-1} P_{\Y_n} = A A^{-1} P_{\Y_n} = P_{\Y_n} = P_{\range{AP_{\U_n}}}$.
	\end{itemize}
	Since the Moore-Penrose equations uniquely characterise the Moore-Penrose inverse, the assertion follows.
\end{proof}
Combination of \eqref{eq:udn1} and Theorem \ref{thm:pseudoinv_image} shows that the Moore-Penrose inverse of \eqref{AP_n u = y} is given by 
\begin{equation}\label{eq:mp}
	\udn = A^{-1}P_{\Y_n} y.
\end{equation}

\paragraph{Gram-Schmidt orthogonalisation in $\Y$.} 

We start by applying the Gram-Schmidt process to the outputs $\trdata$ to obtain an orthonormal basis of $\Y_n$ \rb{(if the training data were not linearly independent, this will be detected by the Gram-Schmidt algorithm and the redundant data will be dismissed, effectively reducing the size of the training set $n$)}. We denote this basis by $\trdatab$. By solving $A u=\orthdata y^i$ for $i=1,\dots,n$, we obtain, in general, a non-orthogonal basis $\trimb$ of $\U_n$. In a matrix form, we can write that
\begin{equation}\label{eq:Y_bar_and_U_bar}
Y_n = \orthdata Y_n R_n \qand \orthdata U_n = U_n R_n^{-1},
\end{equation}
where $Y_n$, $\orthdata Y_n$ and $\orthdata U_n$ \bb{are composed of a finite number of infinite-dimensional functions} $\trdata$, $\trdatab$ and $\trimb$, respectively,
\begin{equation}\label{eq:Y_bar}
Y_n \defeq 
\begin{pmatrix} 
y^1, \dots, y^n
\end{pmatrix}, 
\quad 
\orthdata Y_n = 
\begin{pmatrix} 
\orthdata y^1, \dots, \orthdata y^n
\end{pmatrix},
\quad 
\orthdata U_n = 
\begin{pmatrix} 
\orthdata u^1, \dots, \orthdata u^n
\end{pmatrix}
\end{equation}
and $R_n$ is an upper triangular $n \times n$ transformation matrix
\begin{equation}\label{eq:R_n}
R_n \defeq 
\begin{pmatrix} 
(y^1, \orthdata y^1) & (y^2, \orthdata y^1) & \cdots & (y^n, \orthdata y^1) \\ 
0 			& (y^2, \orthdata y^2) & \cdots & (y^n, \orthdata y^2) \\ 
0			&	0       & \ddots 	& \vdots \\
0 &  0 & \cdots & (y^n, \orthdata y^n)
\end{pmatrix}.
\end{equation}



Taking the noise-free right-hand side $y \in \range{A}$ in~\eqref{eq:Au=y} and expanding $P_{\Y_n}y\in \mathcal{Y}_n$ in the orthonormal basis $\trdatab$, we get
\begin{equation*}
P_{\Y_n} y = \mysum{i}{1}{n} (y, \orthdata y^i) \orthdata y^i,
\end{equation*}  
hence using Theorem~\ref{thm:pseudoinv_image} we can write~\eqref{eq:udn1} as follows
\begin{equation}\label{eq:udn2}
\udn = A^{-1} P_{\Y_n} y = \mysum{i}{1}{n} (y, \orthdata y^i)  \orthdata u^i,
\end{equation}
where $\trimb$ are the transformed inputs satisfying $A \orthdata u^i = \orthdata y^i$, $i = 1,...,n$.

The transformed inputs can be calculated as follows:
\begin{remark}\label{rem:norm_ubar}
	It can be easily verified that the transformed inputs $\trimb$ satisfy
	\begin{equation*}
		\orthdata u^n = \frac{u^n - \sum_{i=1}^{n-1} (y^n, \orthdata y^i) \orthdata u^i}{\norm{y^n - P_{\Y_{n-1}} y^n}}.
	\end{equation*}
	Hence,
	\begin{equation} \label{eq:trim}
		\norm{\orthdata u^n} = \frac{\norm{u^n - \sum_{i=1}^{n-1} (y^n, \orthdata y^i) \orthdata u^i}}{\norm{y^n - P_{\Y_{n-1}} y^n}} \geq \frac{\norm{u^n - P_{\U_{n-1}}u^n}}{\norm{y^n - P_{\Y_{n-1}} y^n}}.
	\end{equation}
\end{remark}

The next remark considers a very particular set of training pairs, namely the singular values of a compact operator. 
With such peculiar training data regularisation by projection becomes a regularisation method, as the following remark shows:
\begin{remark}\label{rem:svd}
Let $A$ be compact with singular value decomposition $\{\sigma^i,x^i,z^i\}_{i \in \N}$, and assume that the 
	training pairs are given by $\{(u^i:=x^i,y^i = \sigma^i z^i = A x^i)\}_{i=1,...,n}$ (see for instance \cite{engl:1996}).
Consequently, $\trdatab = \{z^i\}_{i \in \N}$ is an orthonormal system and accordingly we get $\trimb = \left\{\frac{x^i}{\sigma^i} \right\}$. Then~\eqref{eq:udn2} becomes
\begin{equation*}
\udn = \mysum{i}{1}{n} \frac{1}{\sigma^i} (y, z^i)  x^i = \mysum{i}{1}{n} \frac{1}{\sigma^i} (A u^\dagger, z^i)  x^i = \mysum{i}{1}{n} (u^\dagger, x^i) x^i,
\end{equation*}
i.e. $\udn$ is the projection of the exact solution $u^\dagger$ onto the span of the first $n$ singular vectors $\{x^i\}_{i=1,...,n}$.

In fact this method is a regularisation method (see \cite{engl:1996}).

\end{remark}

\subsection{Behaviour in the limit of infinite data $n \to \infty$}
The 
 nonconvergence example
by Seidman~\cite{Seidman:1980} demonstrates that, {\bf in general}, the minimum norm solution $\udn$ of~\eqref{AP_n u = y} 
{\bf does not converge} to the exact solution $u^\dagger$ as $n \to \infty$,  i.e.  $A^{-1} P_{\Y_n}$ is not a regularisation of $A^{-1}$ (see \eqref{eq:mp}). ``In general'' in the above sentence refers to particular cases such as 
when the training pairs are not singular functions, in which case the method becomes truncated SVD, which is indeed a regularisation method as we have outlined in Remark \ref{rem:svd}. 

Below we analyse the convergence of the Moore-Penrose inverse $\udn$ of \eqref{AP_n u = y} as represented in ~\eqref{eq:udn2} in a 
general setting, in particular, when the training pairs are not spectral pairs as discussed in Remark \ref{rem:svd}.
Letting $n \to \infty$ and applying the Gram-Schmidt process to the sequence $\{y^i\}_{i\in\N}$, we obtain an orthonormal basis $\trdatabinf$  of $\cl{\range{A}}$ and a corresponding sequence $\trimbinf$ such that $A\orthdata u^i = \orthdata y^i$. 

An essential algorithmic step for implementing~\eqref{eq:udn2} is the orthonormalisation of the set $\{y^i\}$, which we perform 
with the Gram-Schmidt algorithm. Therefore we analyse first the stability of this algorithm, which depends on the diagonal 
elements $(y^i, \orthdata y^i)$ of the matrices $R_n$, $i=1,...,n$. It is clear that 
\begin{equation} \label{eq:proj}
\abs{(y^i, \orthdata y^i)} = \norm{y^i - P_{\Y_{i-1}}y^i} \quad \forall i \in \N.
\end{equation}

If $A$ is compact then the part of $y^n$ that lies outside the span of the previous points $\{y^i\}_{i=1,...,n-1}$ will become arbitrary 
small as $n \to \infty$, which shows the instability of the Gram-Schmidt algorithm:
\begin{proposition}\label{prop:proj_yn}
	If $A$ is compact then
	\begin{equation*}
	\liminf_{i\to \infty}\norm{y^i - P_{\Y_{i-1}}y^i} = 0.
	\end{equation*}
\end{proposition}
\begin{proof} Since by Assumption~\ref{ass_1} all training images $\{u^i\}_{i \in \N}$ are uniformly bounded and $A$ is compact, the sequence $\{y^i\}_{i\in \N}$ has a convergent subsequence (that we do not relabel), which, in particular, satisfies
	\begin{equation*}
	\norm{y^i - y^{i-1}} \to 0 \quad \text{as $i \to \infty$}.
	\end{equation*}
	Then we obtain the following estimate
	\begin{equation}\label{eq:est_proj_y^n}
	\norm{y^i - P_{\Y_{i-1}} y^i} = \min_{y \in \Y_{i-1}} \norm{y^i - y} \leq \norm{y^i - y^{i-1}} \to 0,
	\end{equation}
	which proves  the assertion. 
\end{proof}
If the inputs $\trim$ are such that 
\begin{equation*}
\norm{u^n - P_{\U_{n-1}}u^n} \geq C \quad \forall n \in \N,
\end{equation*}
then clearly $\norm{\orthdata u^i} \to \infty$ as $n \to \infty$ by Proposition~\ref{prop:proj_yn} 
	and Remark~\ref{rem:norm_ubar}. 
Hence, the Gram-Schmidt process becomes unstable as $n \to \infty$ and the transformed inputs $\trimb$ may become 
unbounded, as we see from \eqref{eq:proj} and \eqref{eq:R_n}. 
Since $\trimb$ are linearly independent, we can expand the projection $P_{\U_n} u^\dagger$ of the exact solution of~\eqref{eq:Au=y} in the basis of $\U_n$ as follows
\begin{equation}\label{eq:proj_udagger_Un}
P_{\U_n} u^\dagger = \sum_{i=1}^n \alpha_i^n \orthdata u^i, 
\end{equation}
where the expansion coefficients might be varying with respect to $n$, i.e., $\alpha_i^n \neq 
	\alpha_i^m$ for $n \neq m$, since $\trimb$ are not orthogonal. 

The next result shows how far the coefficients of the expansion of the Moore-Penrose approximation $\udn$ in~\eqref{eq:udn2}, 
i.e., $(y, \orthdata y^i)$, deviate from the coefficients $\alpha_i^n$ in~\eqref{eq:proj_udagger_Un} of the best-approximating 
solution $P_{\U_n} u^\dagger$ of $u^\dagger$.
\begin{proposition} Let $P_{\U_n} u^\dagger$ be represented as in \eqref{eq:proj_udagger_Un}, then the following identity holds
\begin{equation*}
\sum_{i=1}^n ( (y, \orthdata y^i) - \alpha_i^n)^2 = \norm{y - AP_{\U_n} u^\dagger}^2 - \norm{y-P_{\Y_n} A u^\dagger}^2.
\end{equation*}
\end{proposition}
\begin{proof}
Consider the residual
\begin{equation*}
u^\dagger - P_{\U_n} u^\dagger = u^\dagger - \sum_{i=1}^n \alpha_i^n \orthdata u^i.
\end{equation*}
Applying $A$ to both sides and expanding $y = A u^\dagger$ in the orthonormal basis $\trdatabinf$, we get
\begin{equation*}
A(u^\dagger - P_{\U_n} u^\dagger) = \sum_{i=1}^\infty (y, \orthdata y^i) \orthdata y^i - \sum_{i=1}^n \alpha_i^n \orthdata y^i = \sum_{i=1}^n ( (y, \orthdata y^i) - \alpha_i^n) \orthdata y^i +  \sum_{i=n+1}^\infty (y, \orthdata y^i) \orthdata y^i
\end{equation*}
and hence
\begin{equation*}
\norm{A(u^\dagger - P_{\U_n} u^\dagger)}^2 = \sum_{i=1}^n ( (y, \orthdata y^i) - \alpha_i^n)^2 + \norm{y-P_{\Y_n} y}^2.
\end{equation*}
Rearranging terms, we get the assertion.
\end{proof}
The first term on the right hand side becomes $0$ if $AP_{\U_n}=P_{\Y_n}A$, which is for instance the case if the training 
	pairs consist of the singular value decomposition (see Remark \ref{rem:svd}). Therefore, the better the subspaces $\U_n$ and $\Y_n$ agree with the spaces spanned by the spectral {functions} of $A$, 
	the smaller the discrepancy between $(y, \orthdata y^i)$ {and} $\alpha_i^n$ for $i,n \in \N$ will be.
%

\subsection{Convergence analysis} \label{conv_an}
In the following we analyse weak and strong convergence of the Moore-Penrose approximation $\udn$ of \eqref{AP_n u = y}, also in the case 
of noisy data.

\subsubsection{Weak convergence}\label{sec:weak}
The first result gives us weak convergence of $\udn$ for $n \to \infty$ in the case of noise free data $y$ (see \eqref{eq:delta}).

\begin{theorem}\label{thm:engl_weak}
Let $y \in \range{A}$ be the exact right-hand side in~\eqref{eq:Au=y} and $\trpairs$ the training pairs defined in~\eqref{eq:pairs}. Let $\lambda_i^n$ be the expansion coefficients of $P_{\Y_n} y \in \Y_n$ in the non-orthogonal basis of training outputs $\trdata$
\begin{equation*}
P_{\Y_n}y = \sum_{i=1}^n \lambda_i^n y^i.
\end{equation*}
Then $\udn$ as defined in~\eqref{eq:udn2} is bounded uniformly in $n$ and converges weakly to $u^\dagger$
\begin{equation*}
\udn \weakto u^\dagger \qquad \text{as $n \to \infty$}
\end{equation*}
if and only if there exists a constant $C_\lambda < \infty$ such that
\begin{equation}\label{old_ass_3}
	 \norm{\mysum{i}{1}{n} \lambda_i^n u^i } \leq C_\lambda \qquad \forall n\in\N.
\end{equation}
\end{theorem}
\begin{proof}
We rewrite~\eqref{eq:udn2} as follows
\begin{equation*}
\udn = A^{-1} P_{\Y_n}y = A^{-1} \sum_{i=1}^n \lambda_i^n y^i=  \sum_{i=1}^n \lambda_i^n u^i,
\end{equation*}
hence $\udn$ is uniformly bounded if and only if~\eqref{old_ass_3} holds. By~\cite[Thm. 3.20]{engl:1996}, $\udn \weakto u^\dagger$ if and only if it is uniformly bounded.
\end{proof}

Below we provide \emph{a priori} conditions on the training data~\eqref{eq:pairs} and the exact solution $u^\dagger$ that ensure boundedness of~\eqref{eq:udn2}. Moreover, we discuss them in the context of Seidman's nonconvergence example~\cite{Seidman:1980} for convergence of regularisation 
by projection. 

\subsubsection{Gram-Schmidt orthogonalisation in $\U$} \label{sec:gs_im}

We will use Gram-Schmidt orthogonalisation again, this time on the inputs $\trim$. We denote the resulting orthonormal basis of $\U_n$ by $\trimhat$. Solving for $\trdatahat$ such that $A \orthim u^i = \orthim y^i$ we obtain, in general, a non-orthogonal basis $\trdatahat$ of $\Y_n$. 

Letting $n \to \infty$ and applying the Gram-Schmidt process to the sequence $\{u^i\}_{i\in\N}$ 
 we obtain an orthonormal basis $\trimhatinf$ of $\U$ and a corresponding sequence $\trdatahatinf$ such that $A \orthim u^i = \orthim y^i$. Expanding the exact solution in the basis $\trimhatinf$, we get
\begin{equation}\label{eq:exp_u_dagger}
u^\dagger = \sum_{i=1}^\infty (u^\dagger, \orthim u^i) \orthim u^i
\end{equation}
and hence for the exact data $y = Au^\dagger$ we get
\begin{equation}\label{eq:exp_y}
y = \sum_{i=1}^\infty (u^\dagger, \orthim u^i) \orthim y^i.
\end{equation}
To check that the latter series is convergent, we note that the partial sums
\begin{equation*}
\sum_{i=1}^n (u^\dagger, \orthim u^i) \orthim y^i = A \sum_{i=1}^n (u^\dagger, \orthim u^i) \orthim u^i
\end{equation*}
are bounded since the operator $A$ is bounded.

Let us recall that we consider two orthonormalisation procedures in this paper, which can be easily confused:
\begin{enumerate}
	\item For $\trpairsb$ from Section~\ref{sec:rec_formula} the outputs $\trdatab$ are orthonormal and the inputs $\trimb$ are chosen to match these outputs, while 
	\item for $\trpairshat$ the inputs $\trimhat$ are orthonormal and the outputs $\trdatahat$ are chosen to match these inputs.
\end{enumerate}
Thus a bar at the bottom stands for orthonormalisation in the range of $A$ and the bar above stands for an orthonormalisation in the domain 
of $A$.

Below we will study the effect of some regularity assumptions of the solution $u^\dagger$ on convergence of $\udn$ for $n \to \infty$.
\begin{assumption}\label{ass:l1_coefs_gt}
Coefficients of the expansion~\eqref{eq:exp_u_dagger} are in $\ell^1$, i.e.
\begin{equation*}
\sum_{i=1}^\infty \abs{(u^\dagger, \orthim u^i)} < \infty.
\end{equation*}
\end{assumption}
\rb{Assumption~\ref{ass:l1_coefs_gt} is a data-driven regularity assumption on the exact solution $u^\dagger$. We show in Section~\ref{sec:seidman} that in Seidman's nonconvergence example Assumption~\ref{ass:l1_coefs_gt} is implied by a source condition. The question whether there is more general relationship between Assumption~\ref{ass:l1_coefs_gt} and source conditions remains open.}

\rb{We emphasize that to prove even weak convergence of regularisation by projection~\eqref{eq:udn2}, additional regularity conditions are unavoidable, since, as demonstrated by Seidman's nonconvergence example, weak convergence fails in general.}
\begin{assumption}\label{ass:l2_coefs_proj}
For every $n \in \N$ and any $i \geq n+1$ consider the following expansion of $P_{\Y_n} \orthim y^i \in \Y_n$
\begin{equation}\label{eq:qn}
P_{\Y_n} \orthim y^i = \sum_{j=1}^n \beta_j^{i,n} \orthim y^j.
\end{equation}
We assume that for every $n \in \N$ 
\begin{equation}\label{eq:assump4}
\sum_{j=1}^n (\beta_j^{i,n})^2 \leq C,\qquad \textrm{for every}\,\, i \geq n+1,
\end{equation}
where $C>0$ is a constant independent of $i$ and $n$. 
\end{assumption}
We extend the definition of $\beta_j^{i,n}$ for $i \leq n$.
\begin{definition} \label{def:4}
	For every $i \leq n$, $\orthim y^i \in \Y_n$
	we define 
	$\beta_j^{i,n} = \delta_{ij}$, 
	where $\delta_{ij}$ is the Kronecker symbol.
\end{definition}
From this definition it follows that
\begin{equation}\label{eq:cons_assump4}
\sum_{j=1}^n (\beta_j^{i,n})^2 = 1,\qquad \textrm{for every}\,\, i\leq n.	
\end{equation}

\begin{theorem}\label{thm:conv_lsq}
Let Assumptions~\ref{ass:l1_coefs_gt} and~\ref{ass:l2_coefs_proj} be satisfied. Then $\udn$ as defined in~\eqref{eq:udn2} is uniformly bounded with respect to $n$.
\end{theorem}
\begin{proof}
Applying $P_{\Y_n}$ to~\eqref{eq:exp_y}, we get
\begin{equation*}
P_{\Y_n} y = \sum_{i=1}^\infty (u^\dagger, \orthim u^i) P_{\Y_n} \orthim y^i
\end{equation*}
and hence, by 
applying $A^{-1}$ to both sides, we get that
\begin{equation*}
\udn = A^{-1}P_{\Y_n}y = \sum_{i=1}^\infty (u^\dagger, \orthim u^i) A^{-1}P_{\Y_n} \orthim y^i.
\end{equation*}
Using H\"older's inequality, we estimate 
\begin{equation*}
\norm{\udn} \leq \sum_{i=1}^\infty \left(\abs{(u^\dagger, \orthim u^i)}\right) \sup_{i=1,...,\infty} \norm{A^{-1}P_{\Y_n} \orthim y^i}.
\end{equation*}
The sum $\sum_{i=1}^\infty \abs{(u^\dagger, \orthim u^i)}$ is bounded by Assumption~\ref{ass:l1_coefs_gt}. We further observe that
\begin{equation*}
A^{-1}P_{\Y_n} \orthim y^i = A^{-1} \sum_{j=1}^n \beta_j^{i,n} \orthim y^j = \sum_{j=1}^n \beta_j^{i,n} \orthim u^j
\end{equation*}
and, since $\trimhat$ are orthonormal, we get that
\begin{equation*}
\norm{A^{-1}P_{\Y_n} \orthim y^i}^2 = \sum_{j=1}^n (\beta_j^{i,n})^2,
\end{equation*}
which is bounded uniformly in $n$ and $i$ by Assumption~\ref{ass:l2_coefs_proj}. Therefore, $\sup_{i=1,...,\infty} \norm{A^{-1}P_{\Y_n} \orthim y^i}$ is bounded uniformly in $n$ and
\begin{equation*}
\norm{\udn} \leq C < \infty \quad \text{uniformly in $n$.}
\end{equation*} 
\end{proof}

\subsubsection{Seidman's nonconvergence example}\label{sec:seidman}
Next we discuss {Seidman's example} from~\cite{Seidman:1980} on nonconvergence of regularisation by projection.
\begin{example}[Seidman~\cite{Seidman:1980}]
Let $\{e^i\}_{i \in \N}$ be any orthonormal basis in $\U$ and $A \colon \U \to \U$ be defined as follows
\begin{equation*}
A \colon \sum_{i=1}^\infty \xi_i e^i \to \sum_{i=1}^\infty(a_i \xi_i + b_i \xi_1) e^i
\end{equation*}
with
\begin{equation} \label{eq:ab}
b_i = 
\begin{cases}
0, \quad & \text{if $i=1$}, \\
i^{-1}, \quad & \text{if $i \geq 2$},
\end{cases}
\qquad 
a_i = 
\begin{cases}
i^{-1}, \quad &\text{if $i$ is odd}, \\
i^{-\frac52}, \quad &\text{if $i$ is even}.
\end{cases}
\end{equation}
This operator is injective and compact (for details see also~\cite{engl:1996}). In our notation, $\trimhatinf = \{e^i\}_{i \in \N}$ (since the inputs are already orthonormalised) and 
\begin{equation*}
\orthim y^i = A \orthim u^i = 
\begin{cases}
\sum_{j=1}^\infty j^{-1} e^j, \quad & i = 1, \\
a_i e^i, \quad  & i \geq 2.
\end{cases}
\end{equation*}
Or in other words our training pairs are $(e^i=\orthim u^i,\orthim y^i)_{i \in \N}$.

For every $n \geq 1$ and every fixed $i \geq n+1$, 
we expand $P_{\Y_n} \orthim y^i \in \Y_n$ as follows
\begin{equation}\label{eq:exp_proj_Seidman}
P_{\Y_n} \orthim y^i = \sum_{k=1}^n \gamma_k \orthim y^k = \gamma_1 \sum_{j=1}^\infty j^{-1} e^j + \sum_{k=2}^n \gamma_k a_k e^k.
\end{equation} 
For all $k=2,...,n$  and $i \geq n+1$  we  therefore  have that
\begin{equation*}
(P_{\Y_n} \orthim y^i, e^k) = (\orthim y^i, P_{\Y_n} e^k) = (\orthim y^i, e^k) = a_i(e^i,e^k) = 0,
\end{equation*}
hence, taking a scalar product with $e^k$ in~\eqref{eq:exp_proj_Seidman}, we conclude that $\gamma_1 k^{-1} + \gamma_k a_k = 0$ and
\begin{equation}	\label{eq:gamma}
\gamma_k = -\gamma_1 \frac{k^{-1}}{a_k}.
\end{equation}
Now, to compute $\gamma_1$, we note that it minimises the following expression for $i \geq n+1$
\begin{eqnarray*}
\norm{P_{\Y_n} \orthim y^i - \orthim y^i}^2 &=& \norm{\gamma_1 \sum_{j=1}^\infty j^{-1} e^j + \sum_{k=2}^n \gamma_k a_k e^k - a_ie^i}^2 \\
&=& \norm{\gamma_1 e^1 + \sum_{k=2}^n (\gamma_1 k^{-1} + \gamma_k a_k)e^k + \gamma_1 \sum_{j=n+1}^\infty j^{-1} e^j - a_ie^i}^2 \\
&\underbrace{=}_{\eqref{eq:gamma}}& \norm{\gamma_1 e^1 + \gamma_1 \sum_{\substack{j=n+1 \\ j \neq i}}^\infty j^{-1} e^j + (\gamma_1 i^{-1} - a_i)e^i }^2 \\
&=& \gamma_1^2 + \gamma_1^2 \sum_{\substack{j=n+1 \\ j \neq i}}^\infty j^{-2} + (\gamma_1i^{-1}-a_i)^2 \\
&=& \gamma_1^2 (1 + \sum_{j=n+1}^\infty j^{-2}) - 2 \gamma_1 i^{-1}a_i + a_i^2.
\end{eqnarray*}
The minimiser of this quadratic expression  with respect to $\gamma_1$  is given by \begin{equation}\label{eq:gamma_1}
\gamma_1 = \frac{i^{-1}a_i}{1+\sum_{j=n+1}^\infty j^{-2}} \defeq C_n i^{-1}a_i,
\end{equation}
where $0 < \frac{1}{1+ \pi^2/6} \leq C_n \defeq 
\frac{1}{1+\sum_{j=n+1}^\infty j^{-2}} \leq 1$ is uniformly bounded from below and above with respect to $n$.

Hence we get from \eqref{eq:gamma} that
\begin{equation}\label{eq:gamma_k}
\gamma_k = -\gamma_1 \frac{k^{-1}}{a_k} = -C_n i^{-1}a_i  \frac{k^{-1}}{a_k}.
\end{equation}
Since $0 < a_i \leq i^{-1}$ for all $i \in \N$ and $\frac{k^{-1}}{a_k} \leq k^{3/2}$ for all $k \in \N$ (see \eqref{eq:ab}), from~\eqref{eq:gamma_1} and~\eqref{eq:gamma_k} we find
\begin{equation*}
	\begin{aligned}
	|\gamma_1|&\leq C_n i^{-2}\leq C_n i^{-\frac{1}{2}},\\
	|\gamma_k|&\leq C_n i^{-2}k^{\frac{3}{2}}\leq C_n i^{-\frac{1}{2}},\qquad \textrm{for}\,\, k=2,\dots,n.
	\end{aligned}
\end{equation*}
Therefore,
\begin{equation*}
\sum_{k=1}^n \gamma_k^2 \leq n C_n^2 i^{-1} \leq C_n^2 \leq 1 \text{ for all } i \geq n+1 
\end{equation*}
and  Assumption~\ref{ass:l2_coefs_proj} is satisfied. 

The nonconvergence example in~\cite{Seidman:1980} is obtained with the following exact solution
\begin{equation*}
u^\dagger = \sum_{i=1}^\infty i^{-1} e^i.
\end{equation*}
Clearly, the expansion coefficients are in $\ell^2$, hence $u^\dagger$ is well defined. However, Assumption~\ref{ass:l1_coefs_gt} fails since 
\begin{equation*}
\sum_{i=1}^\infty i^{-1} = \infty.
\end{equation*}
\end{example}

Therefore, the key to the  nonconvergence example is the slow decay of the expansion coefficients in~\eqref{eq:exp_u_dagger} and hence a violation of Assumption~\ref{ass:l1_coefs_gt}.
This assumption is indeed satisfied in our numerical experiments with natural images (photographs of people) in Section~\ref{sec:numerics}. 

Numerically validating Assumption~\ref{ass:l2_coefs_proj} can be complicated due to numerical errors. Our numerical experiments in Section~\ref{sec:numerics} do not give a definitive answer. However, as the above example demonstrates, this assumption can be satisfied for data generated by a compact operator.

Important for the analysis of regularisation methods are source conditions see \cite{Groetsch:1984} and some references such as \cite{BoeHofTauYam06,FleHofMat11,AndElbHooQiuSch15},
indeed the $l^1$-condition from Assumption~\ref{ass:l1_coefs_gt} is related to the source condition as the following remark shows.
\begin{remark}
In Seidman's example, the $l^1$-condition from Assumption~\ref{ass:l1_coefs_gt} follows from a source condition
\begin{equation}\label{eq:Seidman_sc}
u^\dagger \in \range{A^*}.
\end{equation}
Indeed, suppose that $u = A^*v$ for some $v \in \U$. We get
\begin{eqnarray*}
\sum_{i=1}^\infty \abs{(u,e^i)} &=& \sum_{i=1}^\infty \abs{(A^*v,e^i)} = \sum_{i=1}^\infty \abs{(v,Ae^i)} = \abs{(v,Ae^1)} + \sum_{i=2}^\infty \abs{(v,a_i e^i)} \\ &=& \abs{(v,\orthim y^1)} + \sum_{i=2}^\infty a_i\abs{(v, e^i)}
\leq  \norm{v}\norm{\orthim y^1} + \norm{v} \sqrt{\sum_{i=2}^\infty a_i^2}.
\end{eqnarray*}
Since $a_i \leq i^{-1}$ for all $i \in \N$, we get that
\begin{eqnarray*}
\sum_{i=1}^\infty \abs{(u,e^i)} \leq \norm{v}\norm{\orthim y^1} + \norm{v} \sqrt{\sum_{i=2}^\infty i^{-2}} < \infty.
\end{eqnarray*}

Therefore, for Seidman's nonconvergence example $u^{\dagger}$ violates this source condition, and 
	in particular this means that for obtaining a nonconvergent sequence of regularised solutions $\udn$,  roughness of the solution $u^\dagger$ (meaning not satisfying a source condition) is required. 
\end{remark}
Concluding our discussion of Seidman's nonconvergence example, we recall that the
 nonconvergence of the regularised solution $\udn$ would also not hold if we project onto eigenspaces of the operator $A$ (cf. Remark~\ref{rem:svd}). This is independent of a smoothness assumption on $u^\dagger$.
\subsubsection{Strong convergence}
\rb{Thanks to the assumptions on the operator $A$ and training pairs, expanding the exact data $y \in \range{A}$ of~\eqref{eq:Au=y}  in the
basis of the orthonormalised training outputs $\trdatab$ and formally applying  $A^{-1}$, we obtain the following formal expansion
\begin{equation}\label{eq:formal_series_u_dagger}
u^\dagger = \sum_{i=1}^{\infty}(y, \orthdata y^i) \orthdata u^i,
\end{equation}
where $\trimb$ are the transformed inputs satisfying $A \orthdata u^i = \orthdata y^i$, $i=1,...,n$ (cf. Section~\ref{sec:rec_formula}). 
The residual $u^\dagger - \udn$ is given by
\begin{equation*}
u^\dagger - \udn = \sum_{i=n+1}^{\infty}(y, \orthdata y^i) \orthdata u^i.
\end{equation*}
A sufficient condition for convergence of this residual is absolute convergence of the series~\eqref{eq:formal_series_u_dagger}: }
\begin{equation}\label{eq:abs_conv}
\sum_{i=1}^{\infty}\abs{(y, \orthdata y^i)} \norm{\orthdata u^i} < \infty.
\end{equation}
By Remark~\ref{rem:norm_ubar} we have a lower bound for $\norm{\orthdata u^i}$
\begin{equation}\label{eq:est_norm_ubar_lower}
\norm{\orthdata u^i} = \frac{\norm{A^{-1}(y^i - P_{\Y_{i-1}} y^i)}}{\norm{y^i - P_{\Y_{i-1}} y^i}} \geq \frac{\norm{u^i - P_{\U_{i-1}}u^i}}{\norm{y^i - P_{\Y_{i-1}} y^i}}, \qquad \forall i \geq 2.
\end{equation}
Hence it is clear that for this series to converge, the coefficients $(y, \orthdata y^i)$ would have to decay fast. To understand how fast they need to decay, we need an upper bound on the norm of $\orthdata u^i$. The following result shows that under Assumption~\ref{ass:l2_coefs_proj} the estimate~\eqref{eq:est_norm_ubar_lower} is actually sharp up to a constant factor.
\begin{proposition}\label{prop:est_norm_ubar_upper}
Let Assumption~\ref{ass:l2_coefs_proj} be satisfied.
Then the following estimate holds

\begin{equation*}
\norm{\orthdata u^i} \leq \sqrt{C+1}\frac{\norm{u^i - P_{\U_{i-1}} u^i}}{\norm{y^i - P_{\Y_{i-1}} y^i}}, \qquad \forall\ i \geq 2,  
\end{equation*}
where $C$ is the constant from Assumption~\ref{ass:l2_coefs_proj}.
\end{proposition}
\begin{proof}
First we observe that for all $i \geq 2$ the following identity holds
\begin{equation}\label{eq:proj_yn_ynhat}
y^i - P_{\Y_{i-1}} y^i = (\orthim y^i - P_{\Y_{i-1}} \orthim y^i)\norm{u^i - P_{\U_{i-1}} u^i},
\end{equation}
where $\trpairs$ are the training pairs from~\eqref{eq:pairs} and $\trpairshat$ are the transformed training pairs such that $\trimhat$ are orthonormal and $\trdatahat$ satisfy $A \orthim u^i = \orthim y^i$ (cf. Section~\ref{sec:gs_im}).  Indeed, the Gram-Schmidt process yields that for all $i \geq 2$
\begin{equation*}
\orthim u^i = \frac{u^i - P_{\U_{i-1}} u^i}{\norm{u^i - P_{\U_{i-1}} u^i}} \quad \text{and} \quad \orthim y^i = \frac{1}{\norm{u^i - P_{\U_{i-1}} u^i}}A(u^i - P_{\U_{i-1}} u^i).
\end{equation*}
Hence, we get
\begin{eqnarray*}
(\orthim y^i - P_{\Y_{i-1}} \orthim y^i)\norm{u^i - P_{\U_{i-1}} u^i} &=& A(u^i - P_{\U_{i-1}} u^i) - P_{\Y_{i-1}}A(u^i - P_{\U_{i-1}} u^i) \\
&=& y^i - P_{\Y_{i-1}}y^i - AP_{\U_{i-1}} u^i + P_{\Y_{i-1}}AP_{\U_{i-1}} u^i \\
&=& y^i - P_{\Y_{i-1}}y^i - AP_{\U_{i-1}} u^i + AP_{\U_{i-1}} u^i = y^i - P_{\Y_{i-1}}y^i,
\end{eqnarray*}
where we used the obvious identity $P_{\Y_{i-1}}AP_{\U_{i-1}} = AP_{\U_{i-1}}$. Now, using the expansion in Assumption~\ref{ass:l2_coefs_proj} with $n=i-1$, we get
\begin{equation*}
P_{\Y_{i-1}} \orthim y^i = \sum_{j=1}^{i-1} \beta_j^{i,i-1} \orthim y^j,
\end{equation*}
and hence
\begin{equation*}
A^{-1}(\orthim y^i - P_{\Y_{i-1}} \orthim y^i) = \orthim u^i - \sum_{j=1}^{i-1} \beta_j^{i,i-1} \orthim u^j.
\end{equation*}
Since $\trimhat$ are orthonormal, we get
\begin{equation*}
\norm{A^{-1}(\orthim y^i - P_{\Y_{i-1}} \orthim y^i)}^2 = 1 + \sum_{j=1}^{i-1} (\beta_j^{i,i-1})^2 \leq C+1,
\end{equation*}
where the upper bound follows from Assumption~\ref{ass:l2_coefs_proj}. Using~\eqref{eq:proj_yn_ynhat}, we get
\begin{equation*}
\norm{A^{-1}(y^i - P_{\Y_{i-1}} y^i)} = \norm{A^{-1}(\orthim y^i - P_{\Y_{i-1}} \orthim y^i)}\norm{u^i - P_{\U_{i-1}} u^i} \leq \sqrt{C+1}\norm{u^i - P_{\U_{i-1}} u^i}.
\end{equation*}
Finally, we obtain the desired estimate
\begin{equation*}
\norm{\orthdata u^i} = \frac{\norm{A^{-1}(y^i - P_{\Y_{i-1}} y^i)}}{\norm{y^i - P_{\Y_{i-1}} y^i}} \leq \sqrt{C+1}\frac{\norm{u^i - P_{\U_{i-1}} u^i}}{\norm{y^i - P_{\Y_{i-1}} y^i}}.
\end{equation*} 
\end{proof}

Hence, the sufficient condition~\eqref{eq:abs_conv} 
together with \eqref{eq:proj}
can be written as follows 
\begin{equation}\label{eq:abs_conv2}
\sum_{i=1}^\infty \abs{(y, \orthdata y^i)}\frac{\norm{u^i - P_{\U_{i-1}} u^i}}{\norm{y^i - P_{\Y_{i-1}} y^i}} = \sum_{i=1}^\infty \frac{\abs{(y, \orthdata y^i)}}{\abs{(y^i, \orthdata y^i)}}\abs{(u^i, \orthim u^i)} < \infty,
\end{equation}
where $\trdatab$ are the orthonormalised outputs (cf. Section~\ref{sec:rec_formula}) and $\trimhat$ are the orthonormalised inputs (cf. Section~\ref{sec:gs_im}). Convergence of this series can be ensured either by sufficiently fast decay of the coefficients $(y, \orthdata y^i)$ of the expansion of the exact $y \in \range{A}$ in the orthonormal basis $\trdatabinf$ or by sufficiently fast decay of the scalar products $(u^i, \orthim u^i)$. Note that in $(u^i, \orthim u^i)$,  the index $i$ is present at both positions in the scalar product and $(u^i, \orthim u^i)$ are not expansion coefficients of any particular element of $\U$ in the basis $\trimhatinf$. Rather, $\abs{(u^i, \orthim u^i)} = \norm{u^i - P_{\U_{i-1}} u^i}$ is the part of the training input $u^i$ that lies outside the span of the previous training vectors $\{u^i\}_{i=1,...,n-1}$. \\
\indent For real data where the training inputs $\trim$ and the exact solution $u^\dagger$ are similar, it is likely that $\abs{(y, \orthdata y^i)}$ and $\abs{(y^i, \orthdata y^i)}$ will be of similar order, hence $\frac{\abs{(y, \orthdata y^i)}}{\abs{(y^i, \orthdata y^i)}}$ is unlikely to decay fast. Decay of the scalar products $(u^i, \orthim u^i)$, however, can be controlled by choosing sufficiently similar training inputs $\trim$. Hence, we obtain the following sufficient condition.
\begin{theorem}\label{thm:strong}
Let $y \in \range{A}$ be the exact right-hand side in~\eqref{eq:Au=y}. Suppose that
\begin{enumerate}[label=(\roman*)]
\item \label{thm:strong_i} the expansion coefficients $(y, \orthdata y^i)$ of $y \in \range{A}$ in the orthonormal basis $\trdatabinf$ satisfy
\begin{equation}\label{ass:strong1}
\sup_{i=1,...,\infty} \frac{\abs{(y, \orthdata y^i)}}{\abs{(y^i, \orthdata y^i)}} \leq C < \infty,
\end{equation}
where $C$ is a constant independent of $n$ and $\trdatabinf$ are the orthonormalised outputs,
\item \label{thm:strong_ii}  the training inputs $\triminf$ satisfy
\begin{equation}\label{ass:strong2}
\sum_{i=1}^\infty \abs{(u^i, \orthim u^i)} < \infty,
\end{equation}
where $\trimhatinf$ are the orthonormalised inputs,
\item Assumption~\ref{ass:l2_coefs_proj} is satisfied.
\end{enumerate} 
Then the solution obtained with regularisation by projection~\eqref{eq:udn2} converges strongly to the exact solution of~\eqref{eq:Au=y}, i.e.
\begin{equation*}
\udn \to u^\dagger.
\end{equation*}
\end{theorem}
\begin{proof}
Using H\"older's inequity in~\eqref{eq:abs_conv2}, we get that
\begin{equation*}
\sum_{i=1}^\infty \frac{\abs{(y, \orthdata y^i)}}{\abs{(y^i, \orthdata y^i)}}\abs{(u^i, \orthim u^i)} \leq \left(\sup_{i=1,...,\infty} \frac{\abs{(y, \orthdata y^i)}}{\abs{(y^i, \orthdata y^i)}} \right) \sum_{i=1}^\infty \abs{(u^i, \orthim u^i)} < \infty.
\end{equation*}
This implies~\eqref{eq:abs_conv}, which is sufficient for strong convergence of the series~\eqref{eq:udn2}.
\end{proof}


\begin{remark}
We notice the similarity between the condition~\rb{\eqref{ass:strong2}} and Assumption~\ref{ass:l1_coefs_gt}. Both require $\ell^1$ convergence of certain series; Assumption~\ref{ass:l1_coefs_gt} requires it for $(u^\dagger, \orthim u^i)$, where $u^\dagger$ is the exact solution of~\eqref{eq:Au=y}, whilst~\rb{\eqref{ass:strong2}} requires $\ell^1$ convergence of $(u^i, \orthim u^i)$. If the exact solution $u^\dagger$ and the training inputs $\triminf$ are similar, both conditions are likely to be satisfied or violated simultaneously. Both conditions are satisfied in our numerical experiments with natural images (photographs) in Section~\ref{sec:numerics}. Condition~\rb{\eqref{ass:strong2}} seems also to be satisfied in these experiments. 
\end{remark}

\subsection{Noisy data}\label{sec:direct_noisy}
Even if conditions of Theorem~\ref{thm:strong} are satisfied and we have convergence for clean data $y \in \range{A}$, for noisy data $y^\delta \not\in \range{A}$ convergence may fail.  \bb{Noise in the data $y^\delta$ is understood as a deterministic perturbation $\Delta \not \in \range{A}$ such that $\norm{\Delta} \leq \delta$ and $y_\delta = y + \Delta$.} Define
\begin{equation}\label{eq:udn_delta}
\udndelta \defeq A^{-1}P_{\Y_n} y^\delta. 
\end{equation}
Since $P_{\Y_n} y^\delta \in \range{A}$ for all $n$, this is well defined. Now we get that
\begin{equation*}
u^\dagger - \udndelta = A^{-1} y - A^{-1}P_{\Y_n}(y+\Delta) = A^{-1} (I-P_{\Y_n}) A u^\dagger - A^{-1} P_{\Y_n} \Delta
\end{equation*}
and therefore
\begin{equation}\label{eq:semiconv}
\norm{u^\dagger - \udndelta} \leq  \norm{A^{-1} (I-P_{\Y_n}) A u^\dagger} +  \norm{A^{-1} P_{\Y_n} \Delta}.
\end{equation}
\bb{We  observe the typical \emph{semi-convergence} behaviour.} While the first term converges to zero as $n \to \infty$ under assumptions of Theorem~\ref{thm:conv_lsq}, the second term clearly explodes. The amount of training pairs $n$ therefore plays the role of a regularisation parameter that balances the influence of the two error terms in~\eqref{eq:semiconv}.

\begin{theorem}\label{thm:lsq_par_choice}
Let $n = n(\delta)$ be such that $n(\delta) \to \infty$ as $\delta \to 0$ and
\begin{equation*}
\delta \sqrt{n} \sup_{i=1,...,n} \frac{1}{\abs{(y^i, \orthdata y^i)}} \to 0 \quad \text{as $\delta \to 0$.}
\end{equation*}
Then~\eqref{eq:udn_delta} \bb{together with the function $n(\delta)$ (which we refer to as an a priori parameter choice rule~\cite{engl:1996})} defines a convergent regularisation, i.e.
\begin{equation*}
\udndeltan \to u^\dagger \quad \text{as $\delta \to 0$.}
\end{equation*}
\end{theorem}
\begin{proof}
All we need to do is estimate the growth of $\norm{A^{-1} P_{\Y_n}}$ as $n \to \infty$. For an arbitrary $z \in \Y$ with $\norm{z} \leq 1$ we expand $P_{\Y_n}z \in \Y_n$ in the orthonormal basis $\trdatab$ of $\Y_n$ (see Section~\ref{sec:rec_formula}) and get
\begin{eqnarray*}
\norm{A^{-1}P_{\Y_n}z} &= &\norm{A^{-1}\sum_{i=1}^n (z,\orthdata y^i) \orthdata y^i} = \norm{\sum_{i=1}^n (z,\orthdata y^i) \orthdata u^i} \leq \sqrt{\sum_{i=1}^n (z,\orthdata y^i)^2} \sqrt{\sum_{i=1}^n \norm{\orthdata u^i}^2} \leq \norm{\orthdata U_n}_2,
\end{eqnarray*}
where $\orthdata U_n$ is \bb{composed of the functions} $\trimb$ as defined in~\eqref{eq:Y_bar}. Using~\eqref{eq:Y_bar_and_U_bar}, we further estimate 
\begin{equation*}
\norm{\orthdata U_n}_2 = \norm{U_n R_n^{-1}}_2 \leq \norm{U_n}_2 \norm{R_n^{-1}}_2,
\end{equation*}
where $U_n$ is \bb{composed of the functions} $\trim$ as defined in~\eqref{eq:Y_bar} and $R_n$ is the upper triangular $n \times n$ transformation matrix defined in~\eqref{eq:R_n}. Now, since $R_n$ is upper triangular, its eigenvalues are given by its diagonal entries and $\norm{R_n^{-1}}_2$ is the inverse of the smallest one, i.e.
\begin{equation*}
\norm{R_n^{-1}}_2 = \sup_{i=1,...,n} \frac{1}{\abs{(y^i, \orthdata y^i)}}.
\end{equation*}
We further note that
\begin{equation*}
\norm{U_n}_2 = \sqrt{\sum_{i=1}^n \norm{u^i}^2} = \sqrt{n}
\end{equation*}
by Assumption~\ref{ass_1} and finally we obtain
\begin{eqnarray*}
\norm{A^{-1}P_{\Y_n}z} \leq \sqrt{n} \sup_{i=1,...,n} \frac{1}{\abs{(y^i, \orthdata y^i)}}.
\end{eqnarray*}
The rest follows from~\eqref{eq:semiconv}.
\end{proof}

\section{Dual Least Squares}\label{sec:dual_lsq}
Although projecting the equation~\eqref{eq:Au=y} in   $\U$  as in~\eqref{AP_n u = y} does not yield a convergent solution in general, it is known that projecting~\eqref{eq:Au=y} in  $\Y$ yields convergent solutions. This method is also referred to as \emph{dual least squares}~\cite{engl:1996}.

The dual least squares method consists of finding the minimum norm solution (i.e. least squares solution with minimal norm) of the following problem
\begin{equation}\label{eq:Q_nAu=Q_ny}
P_{\Y_n} A u = P_{\Y_n} y,
\end{equation}
where $P_{\Y_n}$ is the orthogonal projector onto the span of the outputs $\trdata$ and $y \in \range{A}$ is the exact right-hand side in~\eqref{eq:Au=y}. We denote the minimum norm solution of~\eqref{eq:Q_nAu=Q_ny} by $\vdn$, where the superscript $^\Y$ emphasises the fact that the projection in~\eqref{eq:Q_nAu=Q_ny} takes place in $\Y$ (cf. Section~\ref{sec:reg_proj}).

The following result shows that $\vdn$ converges strongly to the exact solution $u^\dagger$ as $n \to \infty$.
\begin{theorem}[{\cite[Thm. 3.24]{engl:1996}}] \label{thm:engl_dual}
Let $y \in \range{A}$ be the exact data in~\eqref{eq:Au=y}. Then the minimum norm solution of~\eqref{eq:Q_nAu=Q_ny} is given by 
\begin{equation}\label{eq:vdn1}
\vdn = P_{A^*\Y_n} u^\dagger,
\end{equation} 
where $P_{A^*\Y_n}$ is the orthogonal projector onto $A^* \Y_n$. Consequently,
\begin{equation*}
\vdn \to u^\dagger \quad \text{as $n \to \infty$}.
\end{equation*}
\end{theorem}

Hence, taking projections in $\Y$ as in~\eqref{eq:Q_nAu=Q_ny} seems a good alternative to projecting in the \rb{input space $\U$} as in~\eqref{AP_n u = y}. However, we will demonstrate that one cannot solve~\eqref{eq:Q_nAu=Q_ny} using training pairs~\eqref{eq:pairs}. Instead, one requires training data for the adjoint $A^*$ of the form 
\begin{equation}\label{eq:pairs_adj}
 \{v^i,y^i\}_{i=1,...,n} \quad \text{such that $v^i = A^* y^i$.}
\end{equation}
It is not clear whether this kind of training data can be obtained in practice, hence the relevance of the dual least squares method in the data driven setting is not clear.

\paragraph{A reconstruction formula.} 
The following result gives a simple characterisation of the Moore-Penrose inverse of $P_{\Y_n}A$, similarly to Theorem~\ref{thm:pseudoinv_image}, and highlights a connection between minimum norm solutions of projected problems~\eqref{AP_n u = y} and~\eqref{eq:Q_nAu=Q_ny}.

\begin{theorem}\label{thm:pseudoinv_data}
Let $A$ have a dense range, $P_{\Y_n}$ as in Definition~\ref{def:Un} and $P_{A^*\Y_n}$ as defined above. Then the Moore-Penrose inverse of $P_{\Y_n}A$ is given by
\begin{equation*}
(P_{\Y_n}A)^\dagger = P_{A^*\Y_n} A^{-1}.
\end{equation*}
Hence, the minimum norm solution $\vdn$ of~\eqref{eq:Q_nAu=Q_ny} is given by
\begin{equation}\label{eq:vdn2}
\vdn = P_{A^*\Y_n} A^{-1} P_{\Y_n} y = P_{A^*\Y_n} \udn,
\end{equation}
where $\udn$ is the minimum norm solution of~\eqref{AP_n u = y} as defined in~\eqref{eq:udn2}.
\end{theorem}
\begin{proof}
First we observe that 
\begin{equation*}
(P_{\Y_n} A)^\dagger = ((A^* P_{\Y_n})^*)^\dagger = ((A^* P_{\Y_n})^\dagger)^* .
\end{equation*}
Since $\range{A}$ is dense in $\Y$, the adjoint $A^*$ is injective, hence we can proceed similarly to Theorem~\ref{thm:pseudoinv_image} and obtain
\begin{equation*}
(A^* P_{\Y_n})^\dagger = (A^*)^\dagger P_{A^*\Y_n} = (A^\dagger)^* P_{A^*\Y_n} = (A^{-1})^* P_{A^*\Y_n},
\end{equation*}
since $A$ is injective. Taking the adjoint, we get that 
\begin{equation*}
(P_{\Y_n} A)^\dagger =  ((A^{-1})^* P_{A^*\Y_n})^* = P_{A^*\Y_n} A^{-1}.
\end{equation*}
The formula~\eqref{eq:vdn2} follows from this expression.
\end{proof}

\begin{remark}
Comparing~\eqref{eq:vdn1} and~\eqref{eq:vdn2}, we notice that
\begin{equation*}
P_{A^*\Y_n} \udn = P_{A^*\Y_n} u^\dagger,
\end{equation*}
i.e. $\udn$ and $u^\dagger$ differ only on the orthogonal complement of $A^*\Y_n$. 
\end{remark}

Since the subspace $A^*\Y_n$ is given by
\begin{equation*}
A^*\Y_n = \span\{A^* y^i\}_{i=1,...,n},
\end{equation*}
to compute the projection $P_{A^*\Y_n}$ one needs to know what $A^* y^i$ are. Hence, training pairs of the form~\eqref{eq:pairs_adj} are required, i.e. one needs training data for the \emph{adjoint} $A^*$. 

Using the orthonormal basis $\trdatab$ of $\Y_n$, we can write the projected equation~\eqref{eq:Q_nAu=Q_ny} as follows
\begin{equation*}
\sum_{i=1}^n (Au, \orthdata y^i)\orthdata y^i  = \sum_{i=1}^n (y,\orthdata y^i) \orthdata y^i
\end{equation*}
and hence
\begin{equation*}
(y,\orthdata y^i) = (Au, \orthdata y^i) = (u, A^* \orthdata y^i) = (u, \orthdata v^i), \quad i=1,...,n,
\end{equation*}
with $\orthdata v^i \defeq A^* \orthdata y^i$, $i=1,...,n$. Hence, $\vdn$ solves the following problem
\begin{equation}\label{eq:dual_lsq}
\min_{u \in \U} \norm{u}^2 \quad \text{s.t. $(u, \orthdata v^i) = (y,\orthdata y^i)$,\quad for $i=1,...,n,$}
\end{equation}
which does not require evaluating $A$ and can be solved using only the training pairs for the adjoint $A^*$~\eqref{eq:pairs_adj}.


\paragraph{Noisy data.}
To complete the presentation of the dual least squares method, we consider reconstructions from noisy data $y^\delta \not\in \range{A}$ such that $\norm{y-y^\delta}\leq\delta$. In this case the size of the training set $n$ also plays the role of the regularisation parameter, as the following result demonstrates.
\begin{theorem}[{\cite[Thm. 3.26]{engl:1996}}]\label{thm:engl_dual_noisy}
Let $y \in \range{A}$ and $y^\delta \in \Y$ s.t. $\norm{y-y^\delta} \leq \delta$. Let $\mu_n$ be the smallest singular value of $P_{\Y_n} A$. If $n = n(\delta)$ is chosen such that $n (\delta) \to \infty$ and $\frac{\delta}{\mu_{n(\delta)}} \to 0$ as $n \to \infty$ then 
\begin{equation} \label{eq:vdn_delta}
\vdndelta \defeq (Q_{n(\delta)}A)^\dagger Q_{n(\delta)} y^\delta 
\to u^\dagger \quad \text{as $\delta \to 0$.}
\end{equation}
\end{theorem}

Non-zero singular values of $P_{\Y_n}A$ coincide with those of $(P_{\Y_n}A)^* = A^*P_{\Y_n}$ and can be computed using the Gram matrix
\begin{equation*}
(A^* \orthdata y^i, A^* \orthdata y^j)_{i,j = 1,...,n} = (\orthdata v^i, \orthdata v^j)_{i,j = 1,...,n},
\end{equation*}
which is available from~\eqref{eq:pairs_adj}, and a parameter choice rule similar to that in Theorem~\ref{thm:lsq_par_choice} can be obtained.

\section{Variational Regularisation}\label{sec:var_reg}
In this section we turn back to the projected problem~\eqref{AP_n u = y} using projections in  $\U$. If assumptions of Theorem~\ref{thm:engl_weak} (or Theorem \ref{thm:conv_lsq})
are not satisfied and the minimum norm solution of~\eqref{AP_n u = y}  may become unbounded as $n \to \infty$, explicit regularisation can be used to ensure boundedness and convergence. A straightforward way to approximate a minimum norm solution of~\eqref{AP_n u = y} is Tikhonov regularisation
\begin{equation}\label{eq:proj_Tikh}
\min_{u \in \U} \frac12 \norm{AP_{\U_n} u - y}^2 + \alpha \norm{u}^2,
\end{equation}
where $y \in \range{A}$ is the exact right-hand side in~\eqref{eq:Au=y}, $P_{\U_n}$ is the projector onto the span of the inputs $\trim$ and $\alpha \in \R_+$ is a regularisation parameter. Clearly, by setting $\alpha = 0$ we obtain a least-squares solution of~\eqref{AP_n u = y}.

More generally, the term $\norm{u}^2$ in~\eqref{eq:proj_Tikh} can be replaced with an arbitrary regulariser $\reg \colon \U \to \R \cup \{+\infty\}$ to obtain the following variational regularisation problem
\begin{equation}\label{eq:proj_var_exact}
\min_{u \in \U} \frac12 \norm{AP_{\U_n} u - y}^2 + \alpha \reg(u).
\end{equation}
Finally replacing the exact data $y$ with a noisy measurement $y^\delta$ such that $\norm{y-y^\delta} \leq \delta$, we obtain
\begin{equation}\label{eq:proj_var_noisy}
\min_{u \in \U} \frac12 \norm{AP_{\U_n} u - y^\delta}^2 + \alpha \reg(u).
\end{equation}

The restricted operator $AP_{\U_n}$ in~\eqref{eq:proj_var_noisy} can be computed \emph{without numerical access} to $A$. Indeed, for any $u \in \U$, we have
\begin{equation*}
P_{\U_n}u = \sum_{i=1}^n (u, \orthim u^i) \orthim u^i \quad \text{and} \quad AP_{\U_n}u = \sum_{i=1}^n (u, \orthim u^i) \orthim y^i,
\end{equation*}
where $\trimhat$ are the orthonormalised inputs and $\trdatahat$ the accordingly transformed outputs, cf. Section~\ref{sec:gs_im}. Therefore, solving~\eqref{eq:proj_var_noisy} does not require direct evaluation of $A$ and can be done just using the training pairs~\eqref{eq:pairs}.
Once the Gram-Schmidt process is complete, evaluating the operator $AP_{\U_n}$ becomes very fast; the main computational burden lies on the orthogonalisation process, which, however, is done ``offline'', i.e. before running any optimisation algorithms on~\eqref{eq:proj_var_noisy}, and only once. Therefore, Gram-Schmidt orthogonalisation can be regarded as ``training''. We note also that adding more training pairs to~\eqref{eq:pairs} does not require retraining, i.e. running the Gram-Schmidt process on the whole training set. Only the new portion of training data needs to be made orthogonal to the old data.

Clearly, the operator $AP_{\U_n}$ approximates $A$ pointwise as $n \to \infty$; if $A$ is compact then we also get approximation in the operator norm~\cite{conway:1985}. Hence,~\eqref{eq:proj_var_noisy} can be regarded as an inverse problem with an inexact forward operator and known results on inverse problems with operator errors can be used (e.g.,~\cite{NeuSch90, Poeschl:2010}). 

\subsection{Convergence analysis}
The goal of this section is to show existence of minimisers in~\eqref{eq:proj_var_noisy} and obtain convergence rates under an appropriate parameter choice rule $\alpha = \alpha(n,\delta)$. There are no new results here as standard results are applicable, the main point being that by formulating~\eqref{eq:proj_var_noisy} we are able to transfer these standard, model-based results into the purely data driven, model-free setting.

We emphasise that the size of the training set in this setting $n$ controls the approximation quality of the forward operator and affects the choice of the regularisation parameter $\alpha$ along with the noise level $\delta$. We formulate the parameter choice rules in terms of the residual $\norm{(I-P_{\U_n})\Jminsol}$ of the expansion of the exact solution in the basis $\trimhatinf$. A more common way of deriving parameter choice rules (e.g.,~\cite{NeuSch90, Poeschl:2010}) is in terms of the approximation error in the operator norm $h_n$ such that $\norm{A-AP_{\U_n}} \leq h_n$, however, this is a global estimate that depends on how well the subspaces $\U_n$ agree with the operator $A$ (the ideal choice would be, obviously, the eigenspaces of $A$ corresponding to the $n$ largest eigenvalues). The residual $\norm{(I-P_{\U_n})\Jminsol}$ is a local quantity that shows how well the subspaces $\U_n$ approximate the particular solution $\Jminsol$ that we are looking for. Hence, even if the global approximation error $\norm{A-AP_{\U_n}}$ is large, convergence can still be fast if the training data~\eqref{eq:pairs} are chosen well for a particular solution $\Jminsol$. In some sense, the choice of the training inputs $\trim$ in~\eqref{eq:pairs} is a way of using \emph{a priori} information about the solution $\Jminsol$ to solve the inverse problem~\eqref{eq:Au=y}.
 
Below we summarise the main assumptions and recall the required existence and convergence results.
 
\begin{assumption}\label{ass:J}
	The regularisation functional $\reg \colon \U \to \R_+ \cup \{+\infty\}$ is proper, convex and lower-semicontinuous. 
\end{assumption}
\begin{assumption}\label{ass:P_n_ker_J}
	    For every $M,\alpha >0$ and every $n \in \N$, the sets 
	    $\{u \in \U \colon \norm{AP_{\U_n} u - y^\delta}^2+\alpha J(u) \leq M\}$ are weakly sequentially compact.
\end{assumption}
\begin{remark}\label{rem:tot_var}
	If $\mathcal{J}$ is the Total Variation, Assumptions \ref{ass:J} and \ref{ass:P_n_ker_J} are satisfied if $AP_{\U_n}:L^2\to L^2$ does not annihilate constant functions, see \cite{AubVes97}.  
\end{remark}
We are now ready to state a-priori bounds for minimisers of \eqref{eq:proj_var_noisy}.
\begin{theorem}\label{thm:J_bounded}
Suppose that Assumptions~\ref{ass:J} and~\ref{ass:P_n_ker_J} are satisfied and the regularisation parameter $\alpha = \alpha(\delta,n)$ is chosen such that
	\begin{equation*}
	\alpha \to 0 \qand \frac{(\delta + \norm{(I-P_{\U_n}) \Jminsol})^2}{\alpha} \to 0 \quad \text{as $\delta \to 0$ and $n \to \infty$},
	\end{equation*}
where $\Jminsol$ is the exact solution of~\eqref{eq:Au=y}. Then~\eqref{eq:proj_var_noisy} admits a minimiser and for every minimiser $\unJ$ there exists a constant $C >0$ independent of $n$ and $\delta$ such that 
	\begin{equation*}
	\reg(\unJ) \leq C \qand \norm{\unJ} \leq C.
	\end{equation*}
\end{theorem}
\begin{proof}
The proof is similar to \cite{NeuSch90, Poeschl:2010}. However, note that in \cite{NeuSch90, Poeschl:2010}, the regularisation 
solution of the functional $u \in U \to \norm{Au-y^\delta}^2+\alpha J(u)$ is approximated by the 
minimiser of the same functional over $\U_n$, while we solve the minimisation on the infinite
dimensional space.
\end{proof}

In modern variational regularisation, (generalised) Bregman distances are typically used to study convergence of approximate solutions~\cite{Benning_Burger_modern:2018}. We briefly recall the definition.
\begin{definition} For a proper convex functional $\reg$ the generalised Bregman distance between $u,v \in \U$ corresponding to the subgradient $p \in \dJ(v)$ is defined as follows
	\begin{equation*}
	D_\reg^p(u,v) \defeq J(u) - J(v) - (p,u-v).
	\end{equation*}
	Here $\dJ(v)$ denotes the subdifferential of $\reg$ at $v \in \U$.
\end{definition}
To obtain convergence rates, an additional assumption on the regularity of the exact solution, called the~\emph{source condition}, needs to be made~\cite{engl:1996}. Several variants of the source condition exist (e.g.,~\cite{engl:1996, scherzer_var_meth:2009, Benning_Burger_modern:2018}); we use the following one~\cite{Burger_Osher:2004}.
\begin{assumption}[Source condition]\label{ass:sc}
	There exists an element $q \in \Y$ such that
	\begin{equation*}
	A^*q \in \dJ(\Jminsol).
	\end{equation*}
\end{assumption}

\begin{theorem}\label{thm:conv_rates_var}
	Suppose that Assumptions~\ref{ass:J},~\ref{ass:P_n_ker_J} and~\ref{ass:sc} are satisfied. 
	Then the following estimate for the Bregman distance between $\unJ$ and $\Jminsol$ corresponding to the subgradient $A^*q$ from Assumption~\ref{ass:sc} holds
	\begin{equation*}
	D_\reg^{A^* q}(\unJ,\Jminsol) \leq  \frac{1}{2\alpha} \left(\delta + \norm{A}\norm{(I-P_{\U_n})\Jminsol} \right)^2 + \frac{\alpha}{2}\norm{q}^2  +   (\delta \norm{q} + C \norm{(I-P_{\U_n})A^*q})
	\end{equation*}
	for some constant $C>0$. 
	
	If the regularisation parameter $\alpha = \alpha(\delta,n)$ is chosen as in Theorem~\ref{thm:J_bounded} then 
	\begin{equation*}
	D_\reg^{A^* q}(\unJ,\Jminsol) \to 0 \quad \text{as $\delta \to 0$ and $n \to \infty$}.
	\end{equation*}
	For the particular choice
	\begin{equation}\label{var:parameter_choice}
	\alpha \sim (\delta + \max\{\norm{(I-P_{\U_n})\Jminsol},\norm{(I-P_{\U_n})A^*q}\})
	\end{equation}
	we obtain the following estimate
	\begin{equation*}
	D_\reg^{A^* q}(\unJ,\Jminsol)  \sim \alpha.
	\end{equation*}
\end{theorem}
\begin{proof}
The proof is similar to \cite{NeuSch90, Poeschl:2010}, with the same caveat as in Theorem~\ref{thm:J_bounded}.
\end{proof}

The fact that the speed of convergence depends on how well the subspaces $\U_n$ spanned by the training inputs $\trim$ approximate the solution $\Jminsol$ is not surprising. However, from Theorem~\ref{thm:conv_rates_var} we conclude that  it is equally important for the convergence rate that the subgradient $A^*q$ from the source condition is well approximated by the subspaces $\U_n$.

\begin{remark}
The results of this section are applicable in a more general setting than the previous parts of the paper. In particular, the forward operator $A$ does not have to be injective. In this case, with some minor (and standard) modifications of the assumptions, we would get convergence of minimisers of~\eqref{eq:proj_var_noisy} to a solution of~\eqref{eq:Au=y} with the minimal value of the regulariser $\reg$ (a \emph{$\reg$-minimising solution}). 
\end{remark}

\section{Numerical Experiments}\label{sec:numerics}
{In this section we present some numerical experiments with the data driven regularisation by projection method based on formula \eqref{eq:udn2}, {the dual least squares method~\eqref{eq:vdn2} and data driven variational regularisation~\eqref{eq:proj_var_noisy}}. 
We explain the data preparation and describe the used data sets first - in particular, how we derive the training pairs for the Radon operator.  }

\subsection{Setting}
\subsubsection{Dataset} \label{sec:dataset}
We use images from ``The 10k US Adult Faces Database''~\cite{faces} (to which we refer as \ds). This is a dataset of $10168$ natural face photographs of $256$ pixels height and variable width (typically between $150$ and $220$ pixels). To see what happens as the size of the training set $n$ gets close to the number of pixels in the image (the discrete analogue of $n \to \infty$), we resize all images to $100 \times 100$ pixels to match the size of our data base ($10k$ images). The effective dimension of the discretised space $\U$, i.e. the number linearly independent images among the $10000$, is smaller than the total dimension, since the images have been cropped with an oval around the face. \bb{About $700$ pixels have the value of exactly $1$ (white) in all images, and further $1500$ have values between $0.9$ and $1$ in all images. This makes the number of ``degrees of freedom'' effectively around $7800$.} Sample images from the database are shown in Figure~\ref{fig:sample_images}.

\begin{figure}[t!]
	\centering
    \begin{minipage}{0.45\textwidth}
	\centering
	\foreach \i in {1,7,10,5} {%
		\begin{subfigure}[p]{0.45\textwidth}
			\includegraphics[width=\textwidth]{faces_\i.png}
			\subcaption{}
		\end{subfigure}
	}
	\caption{Samples from the \ds dataset.}\label{fig:sample_images}
    \end{minipage}\hfill
	\begin{minipage}{0.5\textwidth}
        \centering
        \includegraphics[width=\textwidth]{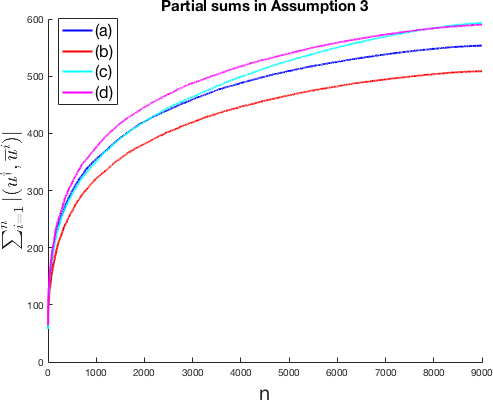}
	\caption{Partial sums $\sum_{i=1}^n \abs{(u^\dagger, \orthim u^i)}$ in Assumption~\ref{ass:l1_coefs_gt} for images $u^\dagger$ shown in Figure~\ref{fig:sample_images} and $n = 1,...,9000$. \rb{The figure does not contradict Assumption~\ref{ass:l1_coefs_gt}}.}\label{fig:l1_coefs_gt}
    \end{minipage}\hfill
\end{figure}

As our training inputs $\trim$ in~\eqref{eq:pairs} we choose $n$ randomly selected photographs, with different $n \leq 10000$. As the ground truth $u^\dagger$ in~\eqref{eq:Au=y} we also take photographs from the \ds dataset (which are not contained in the training set $\trim$).

Applying Gram-Schmidt orthogonalisation to $\trim$ (cf. Section~\ref{sec:gs_im}), we obtain an orthonormal system $\trimhat$. (We use the modified Gram-Schmidt algorithm~\cite{trefethen_num_lin_al}.) Using this system, we can check Assumption~\ref{ass:l1_coefs_gt} from Section~\ref{sec:weak} numerically by plotting the partial sums
\begin{equation*}
\sum_{i=1}^n \abs{(u^\dagger, \orthim u^i)}
\end{equation*}
for different $n$. The result is shown in Figure~\ref{fig:l1_coefs_gt}. The series seem to be bounded uniformly and hence the figure does not contradict Assumption~\ref{ass:l1_coefs_gt}. We stress that Assumption~\ref{ass:l1_coefs_gt} is an assumption on the dataset and does not depend on the forward operator; it only measures how well a single element of the dataset (the exact solution $u^\dagger$) can be approximated with other elements of the dataset (the training inputs $\trim$).

\subsubsection{Gram-Schmidt orthogonalisation vs Householder reflections}\label{sec:GS_numerics}

\rb{It is well know that (even the modified) Gram-Schmidt algorithm is numerically unstable if the number of vectors is large and Householder reflections~\cite{trefethen_num_lin_al} provide a stable alternative. However, the decisive advantage of Gram-Schmidt is that it allows adding new training points without the need to re-orthogonalise the whole dataset, which is clearly advantageous in practice, since orthogonalisation is the most time consuming part of the pipeline. A practical recommendation would be to use Householder reflections to orthogonalise the ``initial'' dataset (say, the training pairs that are available as one starts using the method) and then use Gram-Schmidt to add new training points. Perhaps one could also use Householder reflections once in a while to re-orthogonalise the training set if the number of `new' training points added after the last application of Householder reflections is sufficiently large.}

\rb{However, we would like to emphasise that when regularisation by projection is used with a noisy measurement $y^\delta$ in~\eqref{eq:Au=y} (which is always the case in practice), the amount of training pairs shouldn't be too large anyway as this would compromise stability (cf. Theorems~\ref{thm:lsq_par_choice} and~\ref{thm:engl_dual_noisy} as well as \cite[Thm. 4.2]{Burger_Engl:1999}). This is not the case, however, for variational regularisation, where the noise is counteracted by increasing the regularisation parameter $\alpha$.
}

\subsubsection{Forward operator}\label{sec:forward_op}
We consider our images to be elements of $L^2(\Omega)$, where $\Omega \subset \R^2$ is the (bounded) image domain. As the forward operator, we take the Radon transform with a parallel beam geometry $\Radon \colon L^2(\Omega) \to L^2(\R \times [0,\pi])$. We use Matlab's implementation of the Radon transform \texttt{radon} with $70$ projections uniformly distributed on the interval $[0, \pi)$. The size of the Radon data produced by \texttt{radon} is $145 \times 70 \sim 10k$ pixels. Applying the Radon transform to the training inputs $\trim$, we obtain the outputs (sinograms) $\trdata$ and hence get the pairs~\eqref{eq:pairs}. Sample training pairs are shown in Figure~\ref{fig:faces_triplets} (left and central columns).

\begin{figure}[t!]
	\centering
    \begin{minipage}{0.49\textwidth}
	\foreach \i in {1,7} {%
	 \begin{subfigure}[p]{\textwidth}
	    \includegraphics[width=\textwidth]{Faces_triplet_\i.png}
	 \end{subfigure}
	}
	\caption{Samples $u^i$ from the \ds dataset (left column), their Radon transforms $y^i \defeq Au^i$ (central column) and their adjoint transforms $v^i \defeq A^*y^i = A^*Au^i$ (right column). \\  \\  \\  \\ }\label{fig:faces_triplets}
	\end{minipage}\hfill
	\begin{minipage}{0.49\textwidth}
        \centering
		\includegraphics[width=\textwidth]{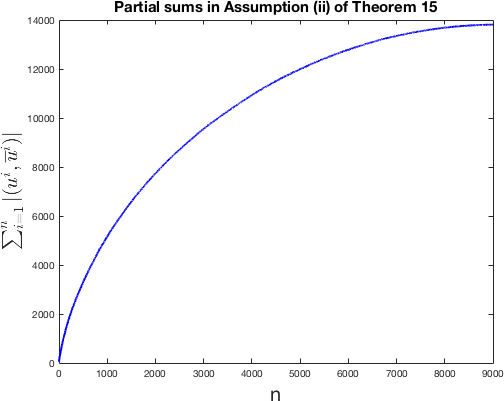}
		\caption{Partial sums in Assumption~\ref{thm:strong_ii} of Theorem~\ref{thm:strong}. \rb{The figure does not contradict Assumption~\ref{thm:strong_ii} of Theorem~\ref{thm:strong}}. The behaviour is similar to that in Figure~\ref{fig:l1_coefs_gt}, but the scale is different ($\sim 14000$ here vs $\sim 600$ in Figure~\ref{fig:l1_coefs_gt}).}
		\label{fig:l1_coefs_strong}
    \end{minipage}
\end{figure}

Assumption~\ref{ass:l2_coefs_proj} depends both on the training inputs $\trim$ and the forward operator $A$. Transforming the training outputs $\trdata$ to match the orthonormalised inputs $\trimhat$, we obtain the transformed sinograms $\trdatahat$. Recall that Assumption~\ref{ass:l2_coefs_proj} was a condition on the expansion coefficients of $P_{\Y_n} \orthim y^i$ for $i \geq n+1$ in the non-orthogonal basis $\trdatahat$. 
To speed up computations, we check this condition on downsampled images of size $32 \times 32$. The result is shown in Figure~\ref{fig:l2_coefs_proj}. The sums are not bounded, hence Assumption~\ref{ass:l2_coefs_proj} doesn't seem to be satisfied. 

However, there is a significant numerical error accumulating in the Gram-Schmidt process as $n$ increases and this error will have an effect on the computed expansion coefficients $\beta_j^{i,n}$. To find these coefficients, we need to solve a linear system whose matrix $\orthim Y_n \defeq (\orthim y^1, \dots, \orthim y^n)$ is formed of the non-orthogonal vectors $\trdatahat$. This matrix will become ill-conditioned for large $n$, since the vectors $\trdatahat$ will become close to being linearly dependent (cf. Proposition~\ref{prop:proj_yn}, where this is proven for $\trdatainf$; the proof for $\trdatahatinf$ is the same). The condition number of this matrix as a function of $n$ is shown in Figure~\ref{fig:l2_coefs_cond_num}; it starts exploding at $n \sim 500$. 
Therefore, in our opinion, the question whether Assumption~\ref{ass:l2_coefs_proj} can be satisfied for real data remains open.

We shall see in Section~\ref{sec:numerics_reg_proj}, however, that the reconstructions~\eqref{eq:udn2} obtained from clean data $y \in \range{A}$ remain bounded as $n$ increases (until numerical artefacts kick in).

Assumptions~\ref{thm:strong_i} and~\ref{thm:strong_ii} of Theorem~\ref{thm:strong} are satisfied. Since the training inputs $\trim$ and the exact solution $u^\dagger$ are similar, we expect that $\frac{\abs{(y, \orthdata y^i)}}{\abs{(y^i, \orthdata y^i)}} ~ \sim 1$, which is in fact what we observe numerically (we omit the plot). Hence, assumption~\ref{thm:strong_i} is satisfied. Assumption~\ref{thm:strong_ii} is also satisfied as shown in Figure~\ref{fig:l1_coefs_strong}.

\begin{figure}[t!]
	\centering
    \begin{minipage}{0.5\textwidth}
	\centering
	\includegraphics[width=\textwidth]{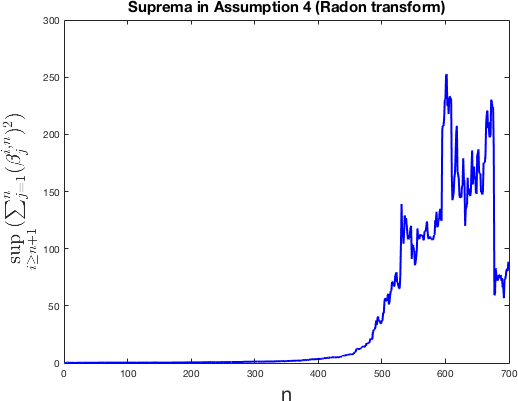}
	\caption{The supremum over the squared $2$-norms of the expansion coefficients of $\orthim y^i$, $i \geq n+1$, in the  basis $\trdatahat$ for $n=1,...,700$ (computed on downsampled images of size $32 \times 32$). Since the basis is non-orthogonal, the expansion coefficients change with $n$. 
		However, we note that the growth can be a numerical artefact: to compute the coefficients $\beta_j^{i,n}$, we need to invert an ill-conditioned matrix whose condition number as a function of $n$ is shown in Figure~\ref{fig:l2_coefs_cond_num}. \\ }
    \end{minipage}\hfill
	\begin{minipage}{0.48\textwidth}
        \centering
	\includegraphics[width=\textwidth]{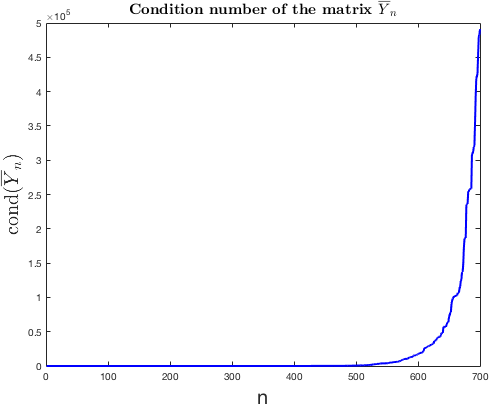}
	\caption{To compute the coefficients $\beta_j^{i,n}$ from Assumption~\ref{ass:l2_coefs_proj}, we need to invert the matrix 
$\orthim Y_n \defeq (\orthim y^1, \dots, \orthim y^n)$ 
composed of the transformed outputs $\trdatahat$, cf. Section~\ref{sec:gs_im}. This matrix becomes ill-conditioned for large $n$, which may result in inaccuracies in the computation of the coefficients $\beta_j^{i,n}$. We emphasise, however, that inverting this matrix is only needed to numerically check Assumption~\ref{ass:l2_coefs_proj}; the actual reconstruction algorithm~\eqref{eq:udn2} does not need it. \\}\label{fig:l2_coefs_cond_num}
    \end{minipage}
\end{figure}


\subsubsection{\bb{Absence of inverse crime}}
\bb{
The term ``inverse crime'' was introduced in~\cite{Colton_Kress:1992} to describe the situation when the same model is used to generate synthetic data and subsequently to solve an inverse problem involving this data. This procedure may artificially decrease the error introduced by a reconstruction algorithm and therefore should be avoided. In our experiments, we use the same model (Matlab's implementation of the Radon transform) to generate both the training data $\trdata$ and the data $y$ that we invert. However, we do not use the same model to solve the inverse problem: we only use the training pairs $\trpairs$ and the approximation of the forward operator or its inverse that we learn from these pairs. Hence, no inverse crime is committed here. Moreover, it is only natural to assume that in real applications the measurement $y$ and the training data $\trdata$ will be generated under very similar conditions, perhaps even by the same measurement device.
}

\subsection{Regularisation by projection}\label{sec:numerics_reg_proj}
We start by analysing reconstructions from clean data $y$ from \eqref{eq:delta}. Figure~\ref{fig:norm_rec_udn} shows the norm of $\udn$ defined in~\eqref{eq:udn2} as a function of $n$ (blue solid line). After some oscillations at small $n$, the norm remains stable until $n \sim 7000$ and then it explodes. This may be due to the failure of Assumption~\ref{ass:l2_coefs_proj} and the instability of regularisation by projection, however, this may be also caused by the numerical instability of Gram-Schmidt orthogonalisation, or by both. We will discuss this in more detail in Section~\ref{sec:numerics_dual}. In any case, boundedness is satisfied for a wide range of $n$. Reconstructions for different values of $n$ are shown in Figure~\ref{fig:udn_clean}. As the size of the training set $n$ increases, the reconstructions start developing oscillations, but generally remain stable for a wide range of $n$.

Adding noise to the data $y$ changes the situation. Reconstructions from noisy data $y^\delta$ (cf.~\eqref{eq:udn_delta}) with $1\%$ Gaussian noise are shown in Figure~\ref{fig:udn_noisy}. Reconstructions become unstable much earlier, reflecting the ill-posedness of the problem. Already for $n=5000$ the reconstruction is highly oscillatory. Figure~\ref{fig:norm_rec_udn} (red dashed line) shows that the norm of the reconstruction $\udndelta$~\eqref{eq:udn_delta} grows with $n$ (except for very small $n$ where it oscillates) and around $n = 5000$ it explodes. 

\begin{figure}
	\begin{minipage}{\textwidth}
	\flushleft
	\foreach \i in {1000,2000,5000,7000} {%
		\begin{subfigure}[p]{0.24\textwidth}
			\includegraphics[width=\textwidth]{lsq_clean_\i.png}
			\subcaption{$n = \i$}
		\end{subfigure}
	}
	\caption{Reconstructions using regularisation by projection~\eqref{eq:udn2} from clean data $y \in \range{A}$. Images develop oscillations as the size of the training set $n$ increases, but generally the reconstructions remain stable for a large range of $n$, until $n \sim 7500$ when they explode.}
	\label{fig:udn_clean}
\end{minipage}
\\[5pt]
\begin{minipage}{\textwidth}
	\flushleft
	\foreach \i in {1000,2000,5000} {%
		\begin{subfigure}[p]{0.24\textwidth}
			\includegraphics[width=\textwidth]{lsq_noisy_\i.png}
			\subcaption{$n = \i$}
		\end{subfigure}
	}\hfill
	\caption{Reconstructions using regularisation by projection~\eqref{eq:udn_delta} from noisy data $y^\delta$ ($1 \%$ noise). Reconstructions become unstable for relatively small $n$. The reconstruction for $n=7000$ is completely unstable and not shown here.}
	\label{fig:udn_noisy}
\end{minipage}
\end{figure}

Figure~\ref{fig:lsq_n_delta} shows the relative reconstruction error as a function of the size of the training set $n$, averaged over a validation set of about $100$ images. The reconstruction error from clean data ($\delta=0$) first decreases monotonically as expected, but has an anomaly around $n=5000$. This may again be due to numerical artifacts of the Gram-Schmidt algorithm. For all other noise levels $\delta = 10^{-4};10^{-3};10^{-2};10^{-1}$ we observe the \emph{semiconvergence} behaviour, i.e. the reconstrcution error first decreases with $n$ until a certain optimal value and then explodes.  This behaviour is expected from Section~\ref{sec:direct_noisy}.

\begin{figure}[h]
	\centering
    \begin{minipage}{0.49\textwidth}
	\centering
	\includegraphics[width=\textwidth]{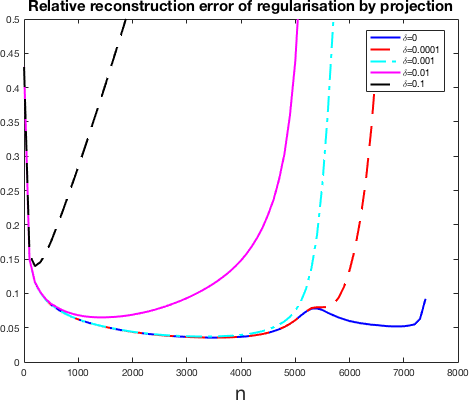}
	\caption{Relative error of regularisation by projection as a function of the size of the training set $n$ for different noise levels $\delta$. The error has been averaged over a validation set of about $100$ images. All curves demonstrate semi-convergence behaviour; the larger the noise, the smaller the optimal $n$ and the larger the error.  	 \\ \\ }\label{fig:lsq_n_delta}
    \end{minipage}\hfill
	\begin{minipage}{0.49\textwidth}
        \centering
		\includegraphics[width=\textwidth]{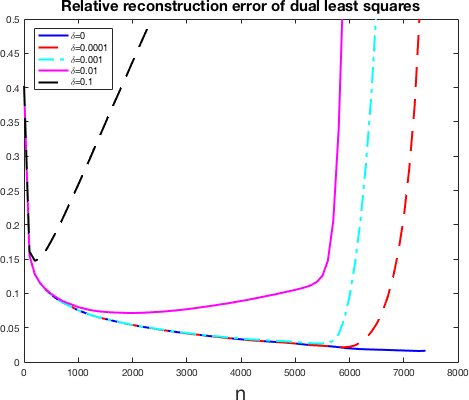}
		\caption{Relative error of dual least squares as a function of the size of the training set $n$ for different noise levels $\delta$. The error has been averaged over a validation set of about $100$ images. All curves except for $\delta=0$ demonstrate semi-convergence behaviour; the larger the noise, the smaller the optimal $n$ and the larger the error. Reconstruction error from clean data ($\delta=0$) decreases monotonically for all $n$.}
		\label{fig:dual_n_delta}
    \end{minipage}
\end{figure}

These experiments demonstrate that for ill-posed problems increasing the size of the training set can have an adverse effect on the reconstructions.

\subsection{Dual least squares}\label{sec:numerics_dual}
The dual least squares approach requires a different type of training data -- training data for the adjoint operator~\eqref{eq:pairs_adj}. It is not clear how this type of training data can be obtained in practice and hence the relevance of the dual least squares for learning is not obvious, but we still perform experiments with this approach for the sake of completeness. To generate $\{v^i\}_{i=1,...,n}$ such that $v^i = A^* y^i$ for $\trdata$ as defined in~\eqref{eq:pairs}, we apply the adjoint of the Radon transform to $\trdata$. Sample training triplets $(u^i,y^i,v^i)$ are shown in Figure~\ref{fig:faces_triplets}.

\begin{figure}[t!]
	\centering
    \begin{minipage}{0.49\textwidth}
    	\centering
		\includegraphics[width=\textwidth]{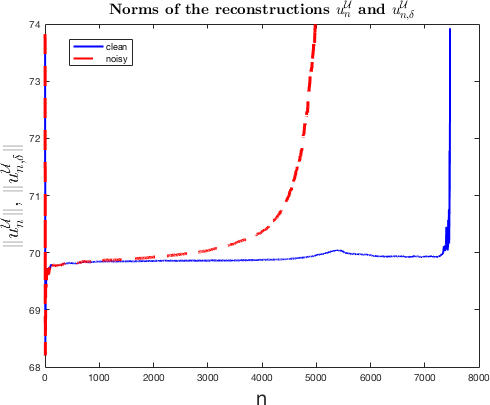}
		\caption{Regularisation by projection: the norm of reconstructions from clean data $y \in \range{A}$ and from noisy data $y^\delta$, denoted by $\udn$~\eqref{eq:udn2} and $\udndelta$~\eqref{eq:udn_delta}, respectively, as a function of $n$. After some oscillations for small $n$, the norm of $\udn$ (clean data) stays almost constant until $n \sim 7500$. 
\bb{The blow-up observed later is caused by the common white oval in all training images. The number of white pixels (with value $1$) in the oval is ca. $700$, and ca. $1500$ more have values between $0.9$ and $1$, which means that the number of linearly independent pixels in the training images is around $7800$.} } 
		\label{fig:norm_rec_udn}
	\end{minipage}\hfill
	\begin{minipage}{0.49\textwidth}
        \centering
		\includegraphics[width=\textwidth]{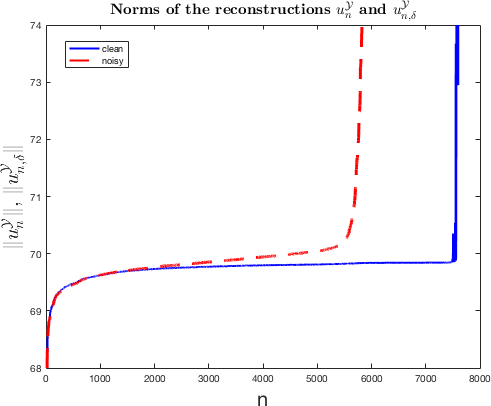}
		\caption{Dual least squares: the norm of reconstructions from clean data $y \in \range{A}$ and from noisy data $y^\delta$, denoted by $\vdn$~\eqref{eq:vdn2} and $\vdndelta$~\eqref{eq:vdn_delta}, respectively, as a function of $n$. The norm of $\vdn$ (clean data) grows monotonically in the beginning and then stays almost constant. At $n \sim 7500$ it explodes, \bb{which is caused by the common white oval in all training images. The number of white pixels (with value $1$) in the oval is ca. $700$, and ca. $1500$ more have values between $0.9$ and $1$, which means that the number of linearly independent pixels in the training images is around $7800$.} }
		\label{fig:norm_rec_vdn}
    \end{minipage}
\end{figure}

We start with reconstructions from clean data $y \in \range{A}$. The norm of the reconstruction~\eqref{eq:vdn2} as a function of the size of training set $n$ is shown in Figure~\ref{fig:norm_rec_vdn} (solid blue line). After a small initial period of monotonic growth, the norm stays constant until $n \sim 7500$ when it explodes. This is the same value as for the solution of regularisation by projection $\udn$ in Figure~\ref{fig:norm_rec_udn}. 

To understand why this happens let us recall that we compute $\vdn$ by projecting the solution of regularisation by projection $\udn$ onto some subspace (cf.~\eqref{eq:vdn2}). Since the dual least squares solution $\vdn$ must be stable as $n \to \infty$ (Theorem~\ref{thm:engl_dual}), the instability must come from an error in computing $\udn$ (cf.~\eqref{eq:udn2}), which is likely to be caused by the numerical instability of the Gram-Schmidt algorithm. This is also an indication that the instability of the solution of regularisation by projection $\udn$ that we see in Figure~\ref{fig:norm_rec_udn} is likely to be due to numerical effects rather than the instability of regularisation by projection.

\begin{figure}
	\begin{minipage}{\textwidth}
	\flushleft
	\foreach \i in {1000,2000,5000,7000} {%
		\begin{subfigure}[p]{0.24\textwidth}
			\includegraphics[width=\textwidth]{dual_clean_\i.png}
			\subcaption{$n = \i$}
		\end{subfigure}
	}
	\caption{Reconstructions using dual least squares~\eqref{eq:vdn2} from clean data $y \in \range{A}$.  Reconstructions remain stable and converge to the ground truth as $n \to \infty$.}
	\label{fig:vdn_clean}
\end{minipage}
\\[5pt]
\begin{minipage}{\textwidth}
	\flushleft
	\foreach \i in {1000,2000,5000,6000} {%
		\begin{subfigure}[p]{0.24\textwidth}
			\includegraphics[width=\textwidth]{dual_noisy_\i.png}
			\subcaption{$n = \i$}
		\end{subfigure}
	}\hfill
	\caption{Reconstructions using dual least squares~\eqref{eq:vdn_delta} from noisy data $y^\delta$ ($1 \%$ noise). We observe a typical semi-convergence behaviour: after improving initially, the reconstructions become unstable as $n \to \infty$ and eventually blow up around $n = 6000$. The reconstruction for $n=7000$ is completely unstable and not shown here.}
	\label{fig:vdn_noisy}
\end{minipage}
\end{figure}

Reconstructions from clean data $y \in \range{A}$ obtained using the dual least squares method~\eqref{eq:vdn2} are shown in Figure~\ref{fig:vdn_clean}. They remain stable as $n$ grows (until numerical instability of the Gram-Schmidt algorithm kicks in at $n \sim 7500$) and converge to the ground truth. They don't develop oscillations that we have seen in Figure~\ref{fig:udn_clean} for regularisation by projection. 

For noisy data $y^\delta$ (Figure~\ref{fig:vdn_noisy}) we observe the expected semi-convergence behaviour: after improving initially, the reconstructions diverge as $n$ increases and blow up somewhere between $n = 5000$ and $n = 6000$. Although both regularisation by projection and dual least squares diverge eventually for noisy data $y^\delta$, the dual least squares method remains stable for a larger $n$, cf. Figures~\ref{fig:norm_rec_udn} vs~\ref{fig:norm_rec_vdn} and Figures~\ref{fig:udn_noisy} vs~\ref{fig:vdn_noisy}.

Figure~\ref{fig:dual_n_delta} shows the relative reconstruction error as a function of the size of the training set $n$, averaged over a validation set of about $100$ images. As expected, the reconstruction error from clean data ($\delta=0$) decreases monotonically for all $n$, while the reconstruction error from noisy data with $\delta = 10^{-4};10^{-3};10^{-2};10^{-1}$ demonstrates the semi-convergence behaviour. The reconstrcution error first decreases with $n$ until a certain optimal value and then explodes.

\subsection{Variational regularisation}\label{sec:numerics_var}
In this section we assess the performance of projected variational regularisation~\eqref{eq:proj_var_noisy}. We only show reconstructions from noisy data $y^\delta$.

For convenience, we repeat the statement of projected variational regularisation~\eqref{eq:proj_var_noisy}
\begin{equation}\label{eq:proj_var_noisy2}
\min_{u \in \U} \frac12 \norm{AP_{\U_n} u - y^\delta}^2 + \alpha \reg(u).
\end{equation}
Our goal is to compare reconstructions for different sizes of the training set $n$ with a model-based reconstruction that has access to the forward operator $A$ (the Radon transform)
\begin{equation}\label{eq:var_model}
\min_{u \in \U} \frac12 \norm{A u - y^\delta}^2 + \alpha \reg(u).
\end{equation}

As in the previous sections, $\U = L^2(\Omega)$, where $\Omega \subset \R^2$ is the image domain. As a prototypical example of a regularisation functional $\reg$ we take Total Variation ($\TV$)~\cite{ROF}, which we define as a functional on $L^2(\Omega)$ extending it with the value $+\infty$ on $L^2(\Omega) \setminus \BV(\Omega)$. This is well defined, since $\Omega \subset \R^2$ and hence $\BV(\Omega) \subset L^2(\Omega)$. This is a common setting in imaging~\cite{Chambolle_Lyons:1997}.

Total Variation is a proper, convex and lower semicontinuous functional on $L^2(\Omega)$~\cite{Acar_Vogel:1994}, hence Assumption~\ref{ass:J} is satisfied. 
	Zeros of the TV functional consist of constant functions. Note that the Radon transform doesn't annihilate the constant functions, hence Assumption \ref{ass:P_n_ker_J} is also satisfied if $P_{\U_n}$ does not annihilate constant functions. For this, it is clearly sufficient that $P_{\U_1} \one \neq 0$, i.e. $(u^1,\one) \neq 0$ (in other words, $u^1$ does not have zero mean). This is clearly satisfied for the photographs in the \ds dataset. Therefore Theorems \ref{thm:J_bounded} and \ref{thm:conv_rates_var} hold.

{To see how well the \emph{learned} operator $AP_{\U_n}$ approximates the Radon transform $A$ when evaluated at the ground truth image $u^\dagger$ (the same as in Figures~\ref{fig:vdn_clean} and~\ref{fig:vdn_noisy}), we show the sinograms $AP_{\U_n}u^\dagger$ (learned) and $Au^\dagger$ (exact) in Figure~\ref{fig:sino_learned_vs_exact}. Already for a moderate size of the training set $n=1000$, the approximation is very good. The relative approximation error as a function of $n$ is shown in Figure~\ref{fig:Radon_learned_err} and seems to decrease exponentially with $n$.}

\begin{figure}[t!]
	\centering
    \begin{minipage}{0.49\textwidth}
    	\centering
		\includegraphics[height=0.92\textwidth]{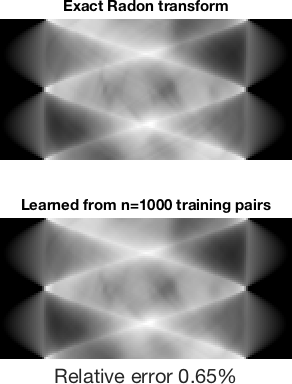}
		\caption{{Top: the sinogram of the image in Figure~\ref{fig:faces_triplets}d. Bottom: learned approximation for $n=1000$ training pairs. Already for a moderate amount of training pairs, the learned operator is able to approximate the Radon transform very well (on an input that is similar to the training inputs).}}
		\label{fig:sino_learned_vs_exact}
	\end{minipage}\hfill
	\begin{minipage}{0.49\textwidth}
        \centering
		\includegraphics[width=\textwidth]{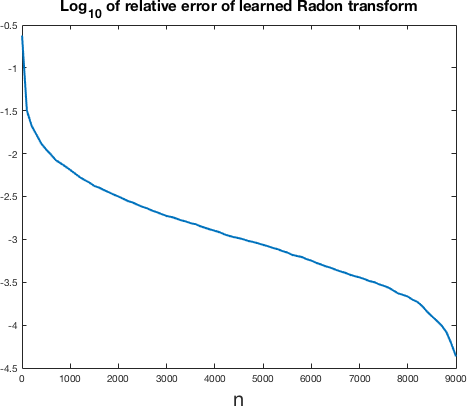}
		\caption{{Reconstruction error of the learned Radon transform as a function of the size of training set $n$ (left: on a linear scale; right: on a $\log$ scale). The error seems to decrease exponentially with $n$.}}
		\label{fig:Radon_learned_err}
    \end{minipage}
\end{figure}

To solve~\eqref{eq:proj_var_noisy2} and~\eqref{eq:var_model}, we use the CVX package~\cite{cvx, cvx2}. To generate the differential operator needed to evaluate $\TV$, we use the DIFFOP package~\cite{diffop}.

Reconstructions from noisy data $y^\delta$ with $1\%$ noise obtained using data driven variational regularisation~\eqref{eq:proj_var_noisy2} are shown in Figures~\ref{fig:var_noisy}\subref{fig:var_1000}--\subref{fig:var_7000} and the solution obtained with the standard model-based approach~\eqref{eq:var_model} is shown in Figure~\ref{fig:var_model}. Already for a moderate size of the training set ($n=1000$, which is $10\%$ of the total number of pixels in each image), we obtain a reasonable reconstruction and as $n$ increases, the reconstructions become closer and closer to the ``ideal'' one obtained using explicit knowledge of the forward model (Figure~\ref{fig:var_model}). The regularisation parameter $\alpha$ was the same in all these experiments.

Data driven reconstructions in Figures~\ref{fig:var_noisy}\subref{fig:var_1000}--\subref{fig:var_7000} exhibit the same qualitative behaviour as the model-based reconstruction in Figure~\ref{fig:var_model}, e.g., we see characteristic for Total Variation \emph{staircasing} that becomes less apparent as $n$ increases.

Relative reconstruction errors for the image shown in Figure~\ref{fig:var_noisy} are given in Table~\ref{tab:var_reg}. For computational reasons, we do not perform experiments on the whole validation set of ca. $100$ images. Although this is not an entirely fair comparison, the reconstruction errors in Table~\ref{tab:var_reg} are smaller than those in Figures~\ref{fig:lsq_n_delta} (regularisation by projection) and~\ref{fig:dual_n_delta} (dual least squares). In the latter two methods the optimal error for $\delta=0.01$ is around $7\%$ (at approximately $n=2000$), whereas for variational regularisation it is smaller even for the same $n$ and keeps decreasing with growing $n$. The difference in the visual quality of the reconstructions is even more apparent, cf. Figures~\ref{fig:udn_noisy},~\ref{fig:vdn_noisy} and~\ref{fig:var_noisy}.

Comparing reconstructions from noisy data $y^\delta$ obtained with data driven variational regularisation (Figures~\ref{fig:var_noisy}\subref{fig:var_1000}--\subref{fig:var_7000}) with those obtained with regularisation by projection (Figure~\ref{fig:udn_noisy}) and dual least squares (Figure~\ref{fig:vdn_noisy}), we observe that variational regularisation produces much better reconstructions. The typical relative error is also smaller, cf. Table~\ref{tab:var_reg} and Figures~\ref{fig:lsq_n_delta} and~\ref{fig:dual_n_delta}. The price to pay for is that variational regularisation~\eqref{eq:proj_var_noisy2} requires solving a (potentially computationally costly) optimisation problem, even if evaluating the projected forward operator $AP_{\U_n}$ becomes cheap once the Gram-Schmidt orthogonalisation is complete, whilst regularisation by projection~\eqref{eq:udn2} and dual least squares~\eqref{eq:vdn2} only require taking a matrix-vector product (also once the Gram-Schmidt algorithm is complete).

\begin{figure}[t!]
	\centering
	\begin{minipage}{\textwidth}
	\foreach \i in {1000,2000,7000} {%
		\begin{subfigure}[p]{0.24\textwidth}
			\includegraphics[width=\textwidth]{var_noisy_\i.png}
			\subcaption{$n = \i$}\label{fig:var_\i}
		\end{subfigure}
	}
	\begin{subfigure}[p]{0.24\textwidth}
			\includegraphics[width=\textwidth]{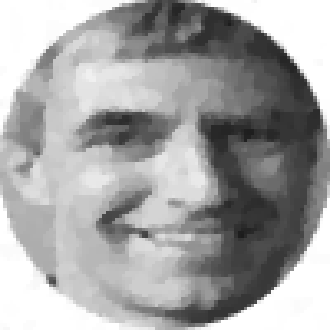}
			\subcaption{Model-based}\label{fig:var_model}
		\end{subfigure}
	\caption{Reconstructions using variational regularisation from noisy data $y^\delta$ ($1\%$ noise): (\subref{fig:var_1000}-\subref{fig:var_7000}) data driven reconstructions~(cf.~\eqref{eq:proj_var_noisy2}) for different sizes of the training set $n$ and (\subref{fig:var_model}) a model based reconstruction  that has access to the forward operator (cf.~\eqref{eq:var_model}). Even for a modest size of the training set $n=1000$ ($10\%$ of the number of pixels in the image) the reconstruction is very reasonable; as $n$ increases, the data driven reconstruction becomes almost indistinguishable from the model-based one.}
	\label{fig:var_noisy}
\end{minipage} \\[10pt]
\begin{minipage}{\textwidth}
\centering
\begin{tabular}{|l|c|c|c|c|c|c|c|c|c|c|}
 \hline
$n$ 		& $1000$		& $2000$		& $3000$		& $4000$		& $5000$		& $6000$		& $7000$		& $8000$		& $9000$ 	& Model \\ \hline
Rel. error & 0.078 & 0.053 & 0.043 & 0.037 & 0.034 & 0.033 & 0.031 & 0.031 & 0.030 & 0.030  \\
\hline
\end{tabular}
\captionof{table}{Relative reconstruction error of data driven variational regularisation for different sizes of the training set $n$ vs. model based reconstruction. Numbers based on the image shown in Figure~\ref{fig:var_noisy}. As $n$ increases, the reconstruction quality approaches that of a model based method. Already for a modest size of the training set $n=3000$ the reconstruction quality is comparable with the optimal one. }\label{tab:var_reg}
\end{minipage}
\end{figure}

\subsection{Practical recommendations}
\rb{Depending on the situation, different methods discussed in this paper will become preferable. If the amount of noise in the measurement $y^\delta$ in~\eqref{eq:Au=y} is small then regularisation by projection (Section~\ref{sec:reg_proj}) is a good option since, once trained, it is very efficient computationally. Care should be taken, however, in checking the validity of our assumption, since this method is non-convergent in general.
}

\rb{The dual least squares method (Section~\ref{sec:dual_lsq}) is better since it does not require additional assumptions (except for an appropriate choice of the size of the training set, cf. Theorem~\ref{thm:engl_dual_noisy}), however, collecting training data for the adjoint operator experimentally is not an obvious task. However, if the goal is to replace a computationally expensive model, then such training data can be collected, and this is the preferred option compared to regularisation by projection.
}

\rb{
Finally, if the amount of noise in the measurement $y^\delta$ is large, parameter choice rules in Theorems~\ref{thm:lsq_par_choice} and~\ref{thm:engl_dual_noisy} will require that the size of the training set is too small and, although stability will hold, the approximation quality will likely be not satisfactory. Variational regularisation is in this case the method of choice, however, it has two drawbacks. Firstly, it is more expensive computationally since it requires solving an optimisation problem, although even here projections might provide a speed-up compared to using an analytic model if this model is computationally expensive. Secondly, a large amount of noise will require a larger regularisation parameter, which will make the effects of the regulariser mode apparent (such as staircasing with Total Variation).
}

\section{Conclusions}
We have seen that some results of model-based regularisation theory can be extended to the purely data driven setting when the forward operator is given only through input-output training pairs. It has also been demonstrated that restrictions of the forward operator and (in the injective case) its inverse to the spans of the training data can be computed without numerical access to the forward operator. This was used to formulate data driven analogues of regularisation by projection and variational regularisation and carry over some classical results such as convergence rates of variational regularisation. We have also seen that the role of the size of the training set is twofold: in variational regularisation, it controls the approximation quality of the forward operator and hence having more training data is always better, while in regularisation by projection, the size of the training set is a regularisation parameter and hence using more training pairs than is allowed by the noise in the measurement will compromise stability. This is due to the ill-posed nature of the inverse problem and different from overfitting, where poor performance is typically a consequence of the training set being too small. The numerical studies should not be considered final. In particular the restrictiveness of Assumption~\ref{ass:l2_coefs_proj} in real world applications needs further studies.

\section*{Acknowledgments}
YK is supported by the Royal Society (Newton International Fellowship NF170045 Quantifying Uncertainty in Model-Based Data Inference Using Partial Order), the Cantab Capital Institute for the Mathematics of Information and the National Physical Laboratory. YK would also like to thank Leon Bungert from the University of Erlangen for stimulating discussions on the topic of this paper. OS is supported by the FWF via the projects I3661-N27 (Novel Error Measures and Source Conditions of Regularization Methods for Inverse Problems) and via SFB F68, project F6807-N36 (Tomography with Uncertainties). 

\printbibliography

\end{document}